\documentclass{amsart}

\usepackage{a4wide}
\usepackage{amsmath, amsfonts, amssymb}
\usepackage{cite}
\usepackage{graphicx}
\usepackage{hyperref}
\usepackage{setspace}
\usepackage{caption}
\usepackage{subcaption}

\usepackage{enumitem}
\usepackage{courier}

\graphicspath{{figs_new/}}

\newcommand{\Lip}{\operatorname{Lip}\nolimits}

\usepackage{color}

\def\h{\lambda}
\def\g{\rho}

\def\ksi{\xi}
\newcommand{\bh}{{\boldsymbol{\lambda}}}

\def\Pcurve{\mathbf{P_{curve}}}
\def\Pmec{\mathbf{P_{mec}}}
\def\PMEC{\mathbf{P_{MEC}}}

\newcommand{\Aut}{\operatorname{Aut}\nolimits}
\newcommand{\smax}{s_{max}}
\newcommand{\smin}{s_{M}}
\newcommand{\tcusp}{t_{cusp}}

\def\th{\theta}

\def\AA{\mathcal{A}}

\newcommand{\spann}{\operatorname{span}\nolimits}
\newcommand{\diag}{\operatorname{diag}\nolimits}
\newcommand{\sign}{\operatorname{sign}\nolimits}
\renewcommand{\Vec}{\operatorname{Vec}\nolimits}

\def\R{{\mathbb R}}
\newcommand{\SO}{\operatorname{SO(3)}\nolimits}
\newcommand{\SOtwo}{\operatorname{SO(2)}\nolimits}
\newcommand{\SEtwo}{\operatorname{SE(2)}\nolimits}
\newcommand{\SE}{\operatorname{SE(3)}\nolimits}

\newcommand{\tx}{{\tilde{x}}}
\newcommand{\ty}{{\tilde{y}}}
\newcommand{\tz}{{\tilde{z}}}
\newcommand{\tal}{\alpha}
\newcommand{\tbe}{\beta}
\newcommand{\tga}{\gamma}
\newcommand{\tpsi}{\tilde{\psi}}

\newcommand{\ulh}[1]{\underline{\h}^{(#1)}}
\newcommand{\be}{{\mathbf{e}}}
\newcommand{\bn}{{\mathbf{n}}}
\newcommand{\bx}{{\mathbf{x}}}
\newcommand{\by}{{\mathbf{y}}}

\newcommand{\bB}{{\mathbf{B}}}
\newcommand{\bN}{{\mathbf{N}}}

\newcommand{\bT}{{\mathbf{T}}}
\newtheorem{theorem}{Theorem}
\newtheorem{cor}{Corollary}
\newtheorem{lemma}{Lemma}

\newtheorem{definition}{Definition}
\newtheorem{remark}{Remark}
\newtheorem{con}{Conjecture}
\newtheorem{corollary}{Corollary}

\numberwithin{equation}{section}

\newcommand{\OSpace}{{\mathbb{R}^3\rtimes S^2}}

\DeclareMathAlphabet\gothic{U}{euf}{m}{n}

\newcommand{\gc}{\gothic{c}}

\newcommand{\bzero}{\mathbf{0}}

\newcommand{\blambda}{{\boldsymbol{\lambda}}}

\newcommand{\dgamma}{{\dot{\gamma}}}

\newcommand{\tR}{{\tilde{R}}}

\newcommand{\cA}{{\mathcal{A}}}
\newcommand{\cC}{{\mathcal{C}}}
\newcommand{\cE}{{\mathcal{E}}}
\newcommand{\cL}{{\mathcal{L}}}

\newcommand{\ul}{\mathbf}
\newcommand{\dif}{\mathrm{d}}
\newcommand{\desda}{\Leftrightarrow}
\begin{document}

\title[sub-Riemannian geodesics in $\SE$]{On sub-Riemannian geodesics in $\SE$ whose spatial projections do not have cusps}

\author{R. Duits}
 \address{Eindhoven University of Technology, The Netherlands}
 \email{r.duits@tue.nl}

 \author{A. Ghosh}
 \address{Link\"oping University, Sweden}
 \email{arpan.ghosh@liu.se}

 \author{T.C.J. Dela Haije}
 \address{Eindhoven University of Technology, The Netherlands}
 \email{t.c.j.dela.haije@tue.nl}

 \author{A. Mashtakov}
 \address{Eindhoven University of Technology, The Netherlands}
 \email{a.mashtakov@tue.nl}

\begin{abstract}
We consider the problem $\Pcurve$ of minimizing $\int \limits_0^L \sqrt{\xi^2 + \kappa^2(s)} \, {\rm d}s$ for a curve $\bx$ in $\R^3$ with fixed boundary points and directions. Here the total length $L\geq 0$ is free, $s$ denotes the arclength parameter, $\kappa$ denotes the absolute curvature of $\bx$, and $\ksi>0$ is constant.
We lift problem $\Pcurve$ on $\R^3$ to a sub-Riemannian problem $\Pmec$ on $\SE/(\{\ul{0}\}\times \SOtwo)$. Here, for admissible boundary conditions, the spatial projections of sub-Riemannian geodesics do not exhibit cusps and they solve problem $\Pcurve$. We apply the Pontryagin Maximum Principle (PMP) and prove Liouville integrability of the Hamiltonian system. We derive explicit analytic formulas for such sub-Riemannian geodesics, relying on the co-adjoint orbit structure, an underlying Cartan connection, and the matrix representation of $\SE$ arising in the Cartan-matrix. These formulas allow us to extract geometrical properties of the sub-Riemannian geodesics with cuspless projection, such as planarity conditions, explicit bounds on their torsion, and their symmetries. Furthermore, they allow us to parameterize all admissible boundary conditions reachable by geodesics with cuspless spatial projection. Such projections lay in the upper half space. We prove this for most cases, and the rest is checked numerically.
Finally, we employ the formulas to numerically solve the boundary value problem, and visualize the set of admissible boundary conditions.
\end{abstract}



\maketitle

\section{Introduction}
\label{intro}
In the space of smooth curves in $\R^3$, we define the energy functional
\begin{equation}\label{EnerFunc}
  \cE(\bx) := \int_0^L \! \sqrt{\xi^2+\kappa^2(s)} \, \dif s,
  \quad \cE : \cC^{\infty}(\R, \R^3) \rightarrow \R^+,
\end{equation}
with $L\in\R^+$ being the length (free) of a curve $s \mapsto \bx(s) \in \R^3$.
Here $\xi>0$ is a constant, $s$ denotes the arclength of the curve $\bx$ and $\kappa : (0,L)\rightarrow \R^+$ denotes the absolute curvature $\kappa(s) = \|\bx''(s)\|$  of the curve $\bx$ for all $s \in (0,L)$.

In this paper we consider the problem $\Pcurve$ of minimizing the functional $\cE(\bx)$ among all smooth curves $s \mapsto \bx(s)$ in $\R^3$, satisfying the boundary conditions (see Figure~\ref{fig:PcurveStat})
\[\bx(0)=\bx_0, \, \bx(L)=\bx_1 \in \R^3, \quad \bx'(0)=\ul{n}_0, \, \bx'(L)=\ul{n}_1 \in \mathrm{S}^2.\]
Here we parameterize $\ul{x}$ by spatial arclength, i.e.
$\|\ul{x}'(s)\|=1$, and via ordinary parallel transport on the tangent bundle $T(\R^{3})$ the tangent vector $\ul{x}'(s) \in T_{\ul{x}(s)}(\R^3)$ can be identified with a point $\ul{n}(s) \in S^{2}$.
\begin{figure}[ht]
\centering
\includegraphics[width=0.4\hsize]{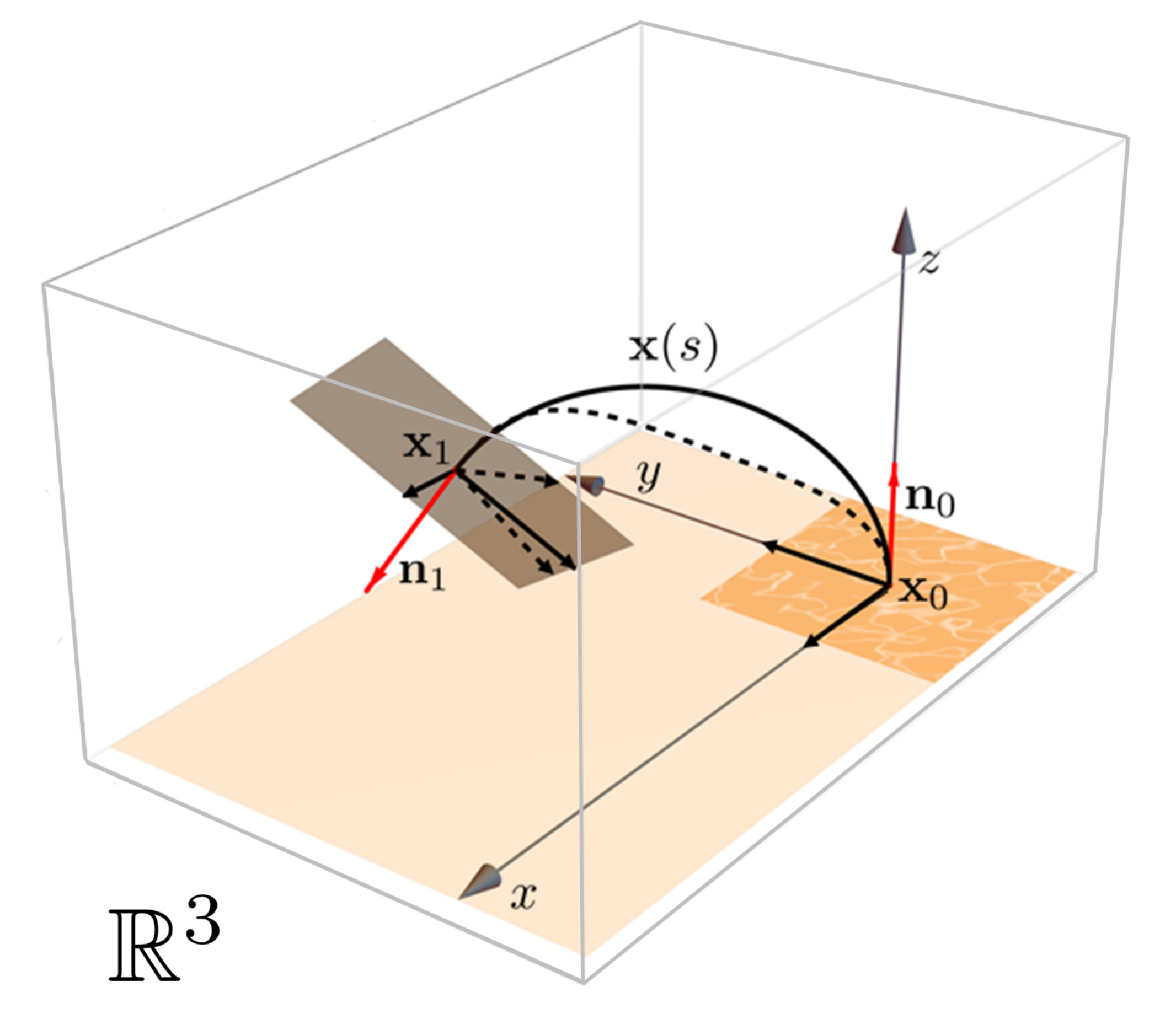}
\includegraphics[width=0.4\hsize]{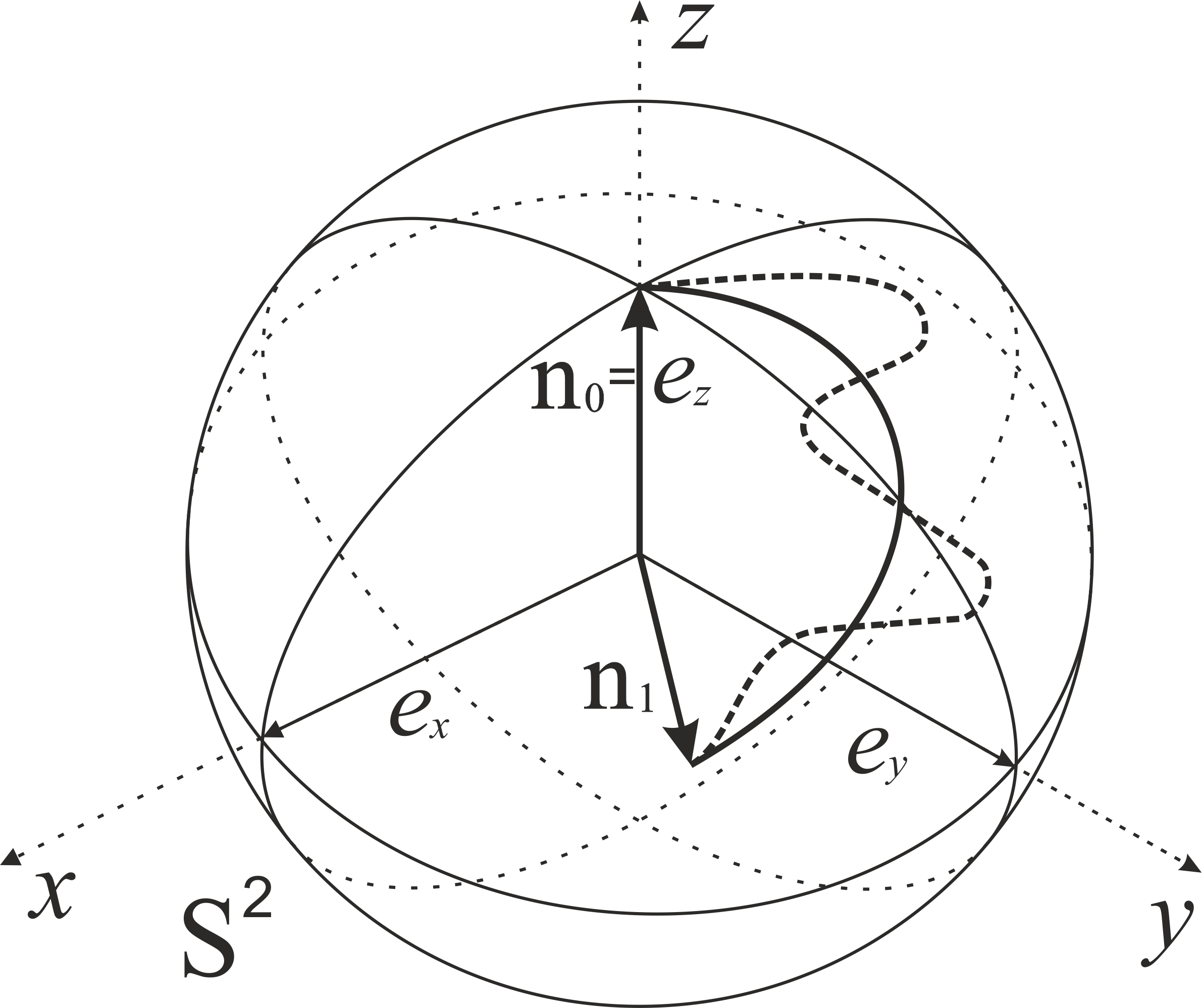}
\caption{Left: Illustration of problem $\Pcurve$. Isotropy in the brown tangent plane spanned by $\{\mathcal{A}_{1},\mathcal{A}_{2}\}$ is needed for a well-posed problem on the Lie group quotient $\SE/(\{\mathbf{0}\}\times \SOtwo)$. The tangent vectors $\bn_0$ and $\bn_1$ are depicted in red.
Right: the angular part $\ul{n}(s)=\ul{x}'(s)$ of the lifted curve $(\ul{x}(s),\ul{x}'(s)) \in \R^{3}\times S^{2}$.}
\label{fig:PcurveStat}
\end{figure}

The two dimensional analog of this variational problem was studied as a possible model of the mechanism used by the visual cortex V1 of the human brain to reconstruct curves which are partially corrupted or hidden from observation. The two dimensional model was initially due to Petitot (see \cite{Petitot, Petitot1999} and references therein). Subsequently,
Citti and Sarti~\cite{Citti, Sanguinetti2008}
 recognized the sub-Riemannian Euclidean motion group structure of the
problem.
In~\cite{Boscain} the existence of minimizers was studied by Boscain, Charlot, and Rossi. It turned out that only for certain end conditions the 2D problem $\Pcurve$ is well-posed. Characterization of the set of end conditions for which $\Pcurve$ is well-posed can be found in~\cite{DuitsSE2}. The more general 2D problem related to a mechanical problem was completely solved by Sachkov~\cite{Sachkov2011,Sachkov,Moiseev}, who in particular derived explicit formulas for the geodesics in sub-Riemannian arclength parameterization. Later, an alternative expression in spatial arclength parameterization for cuspless sub-Riemannian geodesics was derived in \cite{DuitsAMS2,Boscain2013a}. Application of problem $\Pcurve$ to contour completion in corrupted images was studied in~\cite{Mashtakov2013}.
 The problem was also studied by Hladky and Pauls in \cite{Hladky2010} and by Ben-Yosef and Ben-Shahar in \cite{Ben-Yosef2012}. However, many imaging applications such as DW-MRI (Diffusion Weighted Magnetic Resonance Imaging) require an extension to three dimensions \cite{Portegies,R.Duitsa,Franken2009,Haije}, which motivates us to study the three dimensional curves minimizing the energy functional $\cE(\bx)$.
\subsection{Statement of the problem $\Pcurve$}
Let $\bx_0, \bx_1 \in \R^3$ and $\bn_0, \bn_1 \in S^2 = \{\mathbf{v} \in \R^3 | \|\mathbf{v}\| = 1\}$. Our goal is to find an arc-length parameterized curve $s \mapsto \bx(s)$ such that
\begin{equation*}\label{goal1}
    \bx = \arg \inf \limits_{
{\small
        \begin{array}{c}
            \by \in \cC^{\infty}([0,L], \R^3),L \geq 0,     \\
            \by(0) = \bx_0, \; \by'(0) = \bn_0, \\
            \by(L) = \bx_1, \; \by'(L) = \bn_1.
        \end{array}
        }
    } \! \cE(\by).
\end{equation*}
We assume that the boundary conditions $(\bx_0,\bn_0)$ and $(\bx_1,\bn_1)$ are chosen such that a minimizer exists.
Due to rotation and translational invariance of the problem, it is equivalent to the problem with the same functional and boundary conditions $(\bzero,\be_z)$ and $(R_{\ul{n}_0}^T(\bx_1-\bx_0), R_{\ul{n}_0}^T\bn_1)$, where $\be_z$ denotes the unit vector in the $z$-axis in the right handed $\{x,y,z\}$ coordinate system and $R_{\ul{n}_0}\in \SO$ such that $\bn_0 = R_{\ul{n}_0}\be_z$. Therefore, without loss of generality, we set (unless explicitly stated otherwise) $\bx_0 = \bzero$ and $\bn_0 = \be_z$ for the remainder of the article. Hence the problem now is to find a sufficiently smooth arc-length parameterized curve $s \mapsto \bx(s)$ such that
\begin{equation*}\label{goal2}
    \bx = \arg \inf \limits_{
{\small
        \begin{array}{c}
            \by \in \cC^{\infty}([0,L], \R^3),L \geq 0,     \\
            \by(0) = \bzero, \; \by'(0) = \be_z,\\
            \by(L) = \bx_1,  \; \by'(L) = \bn_1.
        \end{array}
   } } \! \cE(\by).
\end{equation*}
We refer to the above problem as $\Pcurve$.

In this paper we use two different parameterizations: spatial arclength $s$ and sub-Riemannian arclength $t(s) = \int_0^s \sqrt{\ksi^2 + \kappa^2(\sigma)} \, {\rm d}\sigma$. We denote the derivative $\frac{d}{d s}$ by a prime, and $\frac{d}{d t}$ by a dot.
\subsection{Structure and Results of the Article}
\ \\
In Section~\ref{ch:2} we lift problem $\Pcurve$ on $\R^3$ to a sub-Riemannian problem $\Pmec$ on the quotient
\begin{equation}\label{OSpace}
\OSpace := \SE/(\{\mathbf{0}\}\times \SOtwo),
\end{equation}
where $\SOtwo$ is identified with all rotations in $\R^3$ about reference axis $\ul{e}_{z}$.
Such an extension and naming (`mec' refers to mechanical) was also done for the problem $\Pcurve$ on $\R^2$, cf.~\cite{DuitsSE2,Boscain2013a}.

To state the problem $\Pmec$ on the quotient~(\ref{OSpace}), we first resort to the corresponding left-invariant sub-Riemannian problem $\PMEC$ on the Lie group $\SE$.
We formulate problem $\PMEC$ in Definition~\ref{def:PMEC} and problem $\Pmec$ in Definition~\ref{def:Pmec}.
\begin{figure}[ht]
\centering
\includegraphics[width=0.7\hsize]{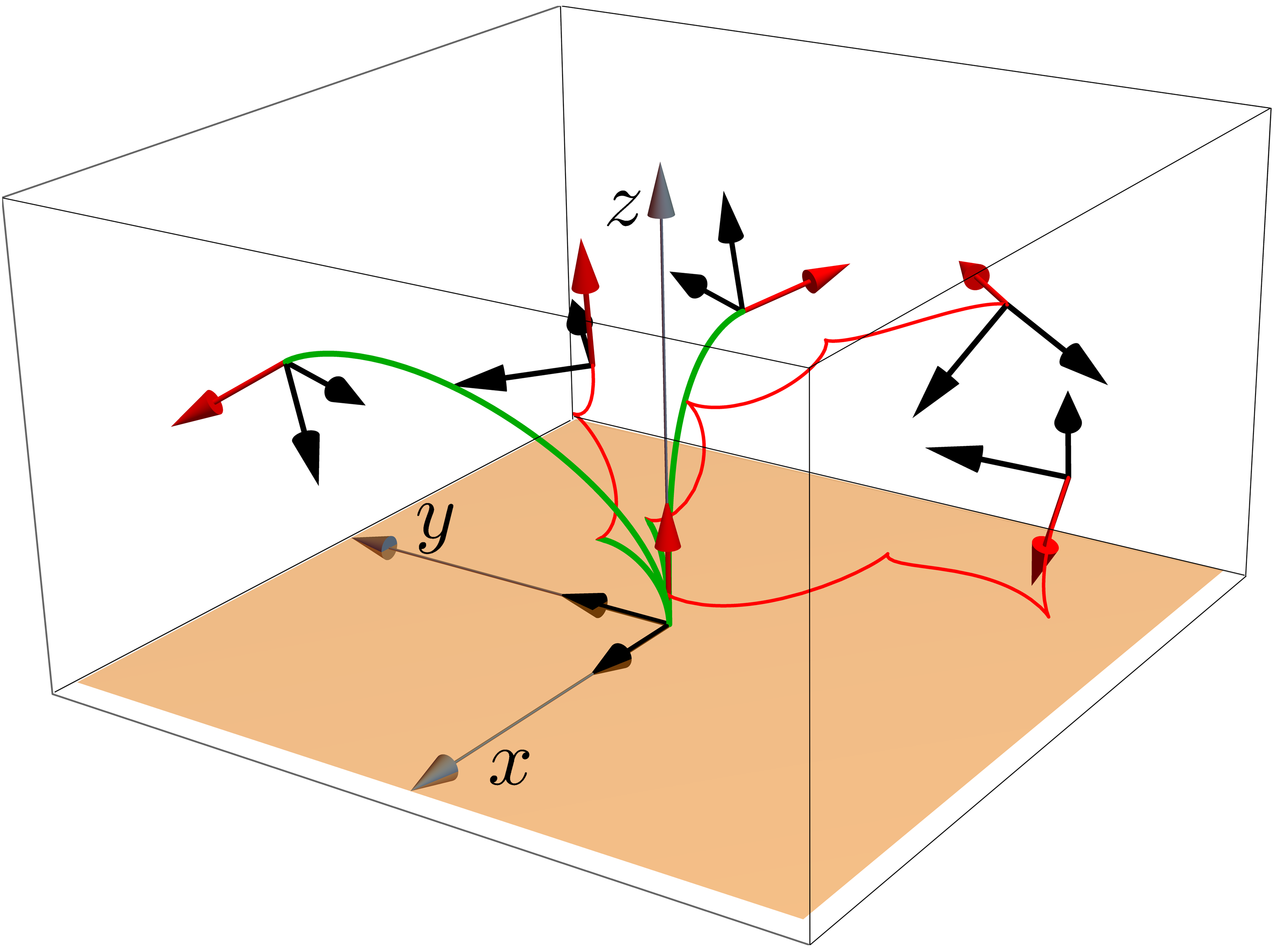}
\caption{The spatial projection of the geodesics of $\Pmec$ can have singularities (the cusp points). Here the spatial projection of the geodesics of $\Pmec$ is shown in green before the first cusp point, and in red after the first cusp point. The range $\mathcal{R}$ of the exponential map of the problem $\Pcurve$ consists of the end conditions reachable by  the cuspless geodesic in $\Pmec$ (i.e. the end conditions reachable by only the green curves). In all cases, the end condition $\bn_1 = R_{\ul{n}_1} \be_z$ is depicted in red. The other black arrows show the remaining vectors $R_{\ul{n}_1} \be_x$ and $R_{\ul{n}_1} \be_y$.}
\label{fig:GeosesicsWithCusps}
\end{figure}

The main result in Section~\ref{ch:2} is Theorem~\ref{lemma:One}, where we show the two requirements for sub-Riemannian geodesics $\gamma(\cdot)=(\ul{x}(\cdot),R(\cdot))$ in $\PMEC$ to have the property that
the corresponding spatially projected curve $\ul{x}(\cdot)$ is indeed a stationary curve of problem $\Pcurve$.
There we also show that these sub-Riemannian geodesics in $\SE$ relate to well-defined geodesics of problem $\Pmec$ on the quotient $\R^{3}\rtimes S^2$.
One of the two requirements is a vanishing momentum component, the other is a requirement on the end-condition $(\ul{x}_{1},\ul{n}_1=R_{\ul{n}_1}\ul{e}_{z})$ which should belong to a
set $\mathcal{R} \subset \OSpace$ that we express as the range of an exponential map of problem $\Pcurve$.
In fact this set $\mathcal{R}$ is precisely the set of end conditions for $\Pmec$ where the spatial projection of geodesics does not exhibit a cusp. The formal definition of a cusp will follow in Definition~\ref{def:cusptime}. For an illustration of cases $(\ul{x}_1,\ul{n}_1)\not\in \mathcal{R}$ where cusps occur on the spatial projection of sub-Riemannian geodesics, see Fig.~\ref{fig:GeosesicsWithCusps}.
The geodesic of $\Pmec$ is said to be \emph{ cuspless} if its spatial projection does not have a cusp. Study of cuspless geodesics in $\SE$ is important for imaging application, namely for tracking of elongated structures in 3D images (see \cite{R.Duitsa, Haije, Portegies}), where presence of cusps is typically undesirable in tracking algorithms. In fact, the presence of a cusp in the spatial projection of a minimizer does not reflect a smooth continuation of local orientations in the 3D images. Likewise to the 2D case~\cite{DuitsSE2} this may even be used as a criterium for not connecting the two boundary conditions.

In Section~\ref{sec:PMEC} we apply the Pontryagin maximum principle (PMP) \cite{notes,Vinter2010,Pontryagin} to problem $\PMEC$ in Subsection~\ref{ch:PMP}.
In Subsection~\ref{subsec:Integrability}, Theorem~\ref{th:integrability} we prove Liouville integrability.
In Subsection~\ref{ch:Cartan} we express the canonical equations of PMP in terms of the $-$ Cartan connection in Theorem~\ref{th:2}. Then a natural choice of $\SE$ matrix representation arises
in the matrix representation of the Cartan connection, i.e. the Cartan-matrix. We employ this in Theorem~\ref{th:3} containing one of the two key ingredients that we use for integrating the canonical equations of $\Pmec$.
The other ingredient is the well-known co-adjoint orbit structure in $\SE$ characterized in Lemma~\ref{lemme:coadjorb}.

In Section~\ref{sec:Pmecs} we combine the two ingredients to compute the first cusp-time (Theorem~\ref{sMax} in Subsection~\ref{ch:cusp}), and to integrate the canonical equations for geodesics of $\Pmec$. As a result, we obtain, for the first time, explicit analytic formulas for both problems $\Pcurve$ and $\Pmec$. These analytic formulas involve elliptic integrals of the first and the third kind.
This is summarized in Theorem~\ref{statcurv}, which is the main result of this article.
Subsequently, in Section~\ref{ch:geom}, we derive many geometric properties of the sub-Riemnnian geodesics such as:
\begin{itemize}
\item a uniform bound on torsion of the spatial part of the geodesics in Theorem~\ref{ThVM},
\item sufficient and necessary conditions for sub-Riemannian geodesics to be planar in Theorem~\ref{coplan} and Corollary~\ref{cor:cop},
for which we have global optimality in Corollary~\ref{cor:embedSe2},
\item monotony along a spatial axis (determined by the initial momentum) in Corollary~\ref{cor:mon},
\item in most cases (see Corollaries~\ref{th:A},~\ref{cor:zpw0},~\ref{th:B})
we prove that the spatial part of sub-Riemannian geodesics stays in the upper half space of the initial direction (if $\ul{n}_{0}=\ul{e}_{z}$ this upper half space
 is $z\geq 0$). In particular we prove $z\geq 0$ for all planar geodesics (Corollary~\ref{cor:zpw0}), and $z\geq 0$ for all geodesics departing from a cusp and ending in a cusp,
\item
in case of planar geodesics and/or geodesics departing from a cusp and ending in a cusp, we show in Corollaries~\ref{th:B}~and~\ref{cor:10} that $z=0$ can only be reached with opposite tangent $-\ul{e}_{z}$ via a U-shaped planar geodesic that departs from a cusp and ends in a cusp,
\item the rotational and reflectional symmetries as we show in Corollary~\ref{AllSym} in Subsection~\ref{ch:symm}.
\end{itemize}

In Section~\ref{ch:num} we conclude with numerical analysis of problem $\Pcurve$ on $\R^3$. Numerical experiments in Subsection~\ref{ch:Jac}, see Figure~\ref{fig:jacobian}, indicate that the first conjugate time comes after the first cusp time, as in the 2D-case \cite{Boscain2013a}. Numerical experiments in Subsection~\ref{ch:conj} on the set $\mathcal{R}$ and the cones of reachable angles, see Figure~\ref{coneslice}, put a conjecture on homeomorphic/diffeomorphic properties on the exponential map (cf.~Conjecture~\ref{Conjecture}).
Finally, we use the analytic formulas for the sub-Riemannian geodesics for numerical solutions to the boundary value problem as briefly explained in Subsection~\ref{ch:BVP}. \emph{Wolfram Mathematica} code for solving the boundary value problem can be downloaded from \url{http://bmia.bmt.tue.nl/people/RDuits/final.rar}.
\section{Problem $\Pcurve$ on $\R^3$, $\PMEC$ on $\SE$, and $\Pmec$ on $\R^3 \rtimes S^2$ and their connection}\label{ch:2}
In this section we relate the problem $\mathbf{P_{curve}}$ to a sub-Riemannian problem $\Pmec$ on the quotient
$\OSpace = \SE/(\{\mathbf{0}\}\times \SOtwo)$,
as was also done for the $\Pcurve$ on $\R^2$, cf. \cite{DuitsSE2,Boscain2013a}. To state the problem $\Pmec$ on this Lie group quotient, we first resort to the corresponding left-invariant sub-Riemannian problem $\PMEC$ in the Lie group $\SE$.
The group $\SE=\R^{3} \rtimes \SO$ denotes the Lie group of rigid body motions on $\R^3$, which is a semi-direct product $\rtimes$ of $\R^{3}$ and $\SO$. An element $g \in \SE$ is represented by the pair $(\bx,R) \in \R^3 \rtimes \SO$, and the group product is given by $g_1 g_2 = (\bx_1,R_1)(\bx_2,R_2) = (\bx_1 + R_1 \bx_2, R_1 R_2)$, and $g^{-1} = (-R^T \bx, R^T)$.
We define sub-Riemannian problem $\PMEC$ by means of the left-invariant frame (see Figure \ref{movingframe}).
\begin{figure}[ht]
\centering
\includegraphics[scale=.25]{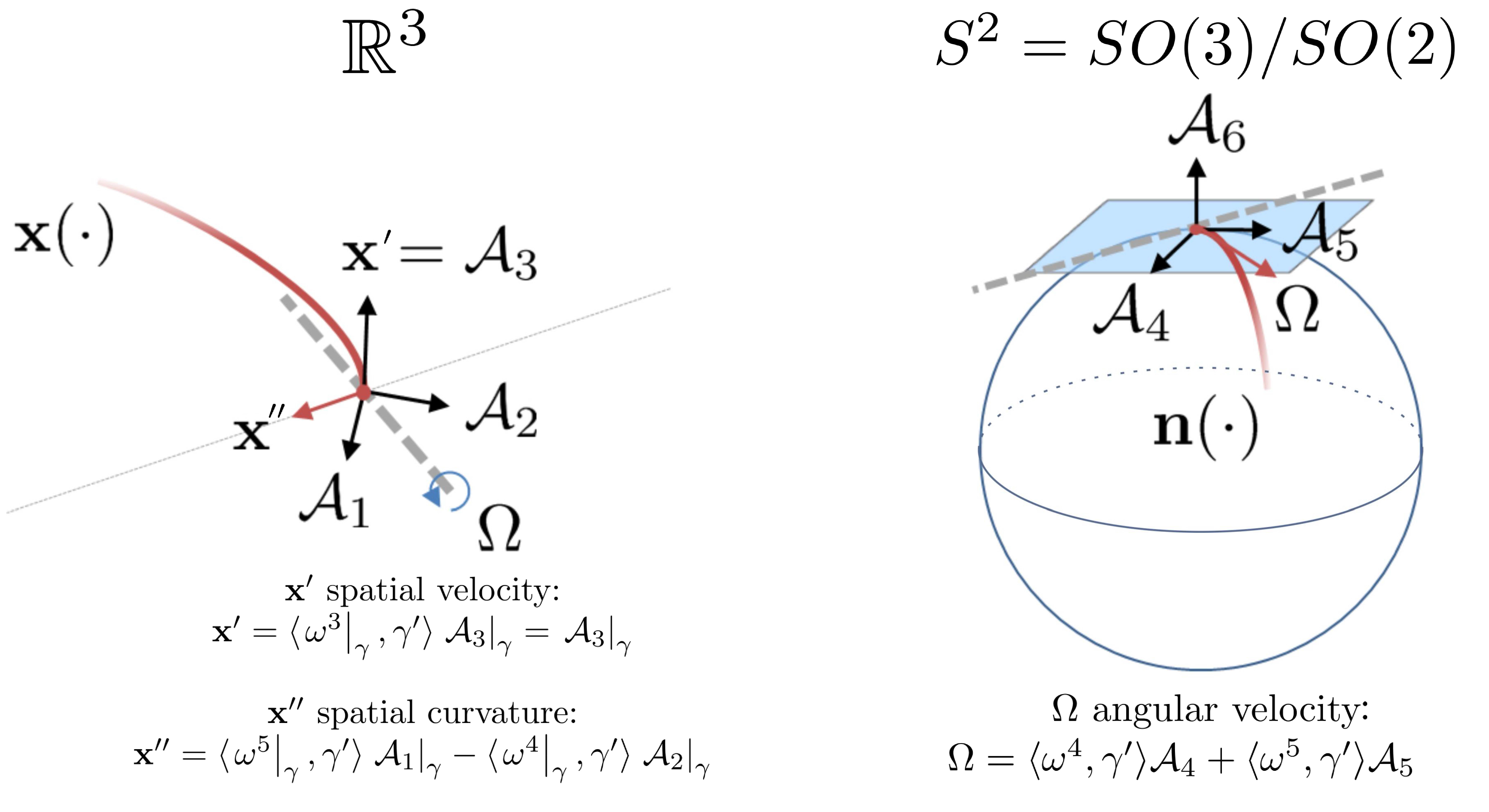}
\caption{Illustrations of the left-invariant frame representing a moving frame of reference along a curve on $\OSpace$.}
\label{movingframe}
\end{figure}

The left-invariant frame consists of the following left-invariant vector fields over $\SE$:
\begin{align*}\nonumber
\cA_1 &= \cos{\tal}\cos{\tbe}\,\partial_x + (\sin{\tal}\cos{\tga} + \cos{\tal}\sin{\tbe}\sin{\tga})\,\partial_y + (\sin{\tal}\sin{\tga} - \cos{\tal}\sin{\tbe}\cos{\tga})\,\partial_z, \\ \nonumber
\cA_2 &= -\sin{\tal}\cos{\tbe}\,\partial_x + (\cos{\tal}\cos{\tga} - \sin{\tal}\sin{\tbe}\sin{\tga})\,\partial_y + (\cos{\tal}\sin{\tga} + \sin{\tal}\sin{\tbe}\cos{\tga})\,\partial_z, \\ \nonumber
\cA_3 &= \sin{\tbe}\,\partial_x - \cos{\tbe}\sin{\tga}\,\partial_y +  \cos{\tbe}\cos{\tga}\,\partial_z, \\ \nonumber
\cA_4 &= -\cos{\tal}\tan{\tbe}\,\partial_{\tal} + \sin{\tal}\,\partial_{\tbe} +  \cos{\tal}\sec{\tbe}\,\partial_{\tga}, \\ \nonumber
\cA_5 &= \sin{\tal}\tan{\tbe}\,\partial_{\tal} + \cos{\tal}\,\partial_{\tbe} -  \sin{\tal}\sec{\tbe}\,\partial_{\tga}, \\ \label{MovFra}
\cA_6 &= \partial_{\tal},
\end{align*}
where we parameterize $\R^3$ by $\{x,y,z\}$ and $\SO$ by angles $\{\tal,\tbe,\tga\}$ with $\tal\in (-\pi,\pi]$, $\tbe\in [-\frac{\pi}{2},\frac{\pi}{2}]$ and $\tga\in (-\pi,\pi]$ such that
\begin{equation} \label{eq:Rparam}
\SO\ni R = \left( \begin{array}{ccc}
            1 & 	0  & 	0          \\
            0 & \cos{\tga} & -\sin{\tga}   \\
            0 & \sin{\tga} &  \cos{\tga}
        \end{array}\right)
	\left( \begin{array}{ccc}
             \cos{\tbe} & 0 & \sin{\tbe}   \\
            	   0 	& 1 & 	0  	   \\
            -\sin{\tbe} & 0 & \cos{\tbe}
        \end{array}\right)
	\left( \begin{array}{ccc}
            \cos{\tal} & -\sin{\tal}  & 0    \\
            \sin{\tal} &  \cos{\tal}  & 0    \\
            	0      & 	0         & 1
        \end{array}\right).
        \end{equation}
Since $R \be_z =\ul{n}=(\sin\tbe, -\sin\tga\cos\tbe, \cos\tga\cos\tbe)^T$ we have that $(\tga,\tbe)$ are spherical coordinates on $S^2$. One needs  multiple charts to cover $S^2$.
However, outside of $\pm (1,0,0)^T$, this choice of spherical coordinates is 1-to-1 and regular, and there it is preferable over standard Euler angles (which are singular at the unity element $\ul{e}_{z}=(0,0,1)^T$), \cite[ch:2,~Fig.4]{DuitsIJCV}.

The corresponding co-frame is given by the co-vectors $\{\omega^1,\ldots,\omega^6\}$ satisfying
$$\langle \omega^i, \cA_j \rangle = \delta^i_j \mbox{ for } i, j \in \{1, \ldots, 6\},$$
with $\delta^i_j$ the Kronecker delta.
The structure constants $c_{i,j}^k$ of the Lie algebra of left-invariant vector fields in $\SE$ are given in Table \ref{tabcomA}.
{\small
\begin{table}[ht]
\begin{center}
\begin{tabular}{c|cccccc}
 & $\cA_1$ & $\cA_2$ & $\cA_3$ & $\cA_4$ & $\cA_5$ & $\cA_6$\\ \hline
$\cA_1$ &0 & 0 & 0 & 0 & $\cA_3$ & $-\cA_2$ \\
$\cA_2$ & 0 & 0 & 0 & $-\cA_3$ & 0 & $\cA_1$ \\
$\cA_3$ & 0 & 0 & 0 & $\cA_2$ & $-\cA_1$ & 0 \\
$\cA_4$ & 0 & $\cA_3$ & $-\cA_2$ & 0 & $\cA_6$ & $-\cA_5$ \\
$\cA_5$ & $-\cA_3$ & 0 & $\cA_1$ & $-\cA_6$ & 0 & $\cA_4$ \\
$\cA_6$ & $\cA_2$ & $-\cA_1$ & 0 & $\cA_5$ & $-\cA_4$ & 0
\end{tabular}
\caption{Table of Lie brackets $[\cA_i,\cA_{j}]=\cA_i \cA_j-\cA_j \cA_i=\sum_{k=1}^6 c^{k}_{i,j} \cA_k$.
\label{tabcomA}}
\end{center}
\end{table}
}

We consider the sub-Riemannian manifold $(M,\Delta,\mathcal{G}_{\ksi})$, \cite{Mont}, with
\begin{equation} \label{SRMAN}
\begin{array}{l}
M=\SE, \ \Delta = \spann \{ \cA_3, \cA_4, \cA_5 \}, \textrm{ and }\mathcal{G}_{\ksi} = \ksi^2 \omega^3\otimes\omega^3 + \omega^4\otimes\omega^4 + \omega^5\otimes\omega^5.
\end{array}
\end{equation}
For those horizontal curves $\gamma$ (i.e. $\dot{\gamma} \in \Delta$) in $\SE$ that can be parameterized by spatial arclength one has $\int_{0}^{T}\!\sqrt{\mathcal{G}_{\ksi}|_{\gamma (t)}(\dot{\gamma}(t),\dot{\gamma}(t))}\,\dif t =  \int_0^L \! \sqrt{\ksi^2+\kappa^2(s)} \, \dif s$, which in view of Problem $\Pcurve$, motivates our choice of $(M,\Delta,\mathcal{G}_{\ksi})$. Details can be found in Appendix~\ref{app:A} (Eq.~\!(\ref{agree})).
The sub-Riemannian distance on $(\SE,\Delta,\mathcal{G}_{\ksi})$ is given by
\begin{equation}\label{eq:CCdistLip}
d(g,h)= \min \limits_{
    \begin{array}{c}
    \gamma \in \Lip([0,T], \SE), T\geq 0, \\
    \dot{\gamma} \in \Delta, \
    \gamma(0)=g, \gamma(T)=h
    \end{array}} \int \limits_{0}^T
    \sqrt{\left. \mathcal{G}_{\ksi}\right|_{\gamma(t)}(\dot{\gamma}(t),\dot{\gamma}(t))}\, {\rm d}t.
\end{equation}
\begin{definition}\label{def:PMEC}
In problem $\PMEC$ on $\SE$, we aim for a Lipschitzian 
curve $\gamma : [0,T]\rightarrow \SE$, that satisfies the boundary conditions $\gamma(0)=e:=(\bzero,I)$ and $\gamma(T)=(\bx_1,R_1) \in \SE$, and minimizes the integral of sub-Riemannian length
$\int_{0}^{T}\!{\sqrt{\mathcal{G}_{\ksi}|_{\gamma (t)}(\dot{\gamma}(t),\dot{\gamma}(t))}\,\dif t}$ 
(with free $T$), 
  and has a velocity vector $\dgamma (t) \in \Delta \textrm{ for a.e. } t \in [0,T].$ 
\end{definition}
\begin{remark}
As we will show in Section~\ref{sec:PMEC}, $\PMEC$ is well-posed and the minimizers are smooth. So one may replace $\Lip([0,T], \SE)$ by $\cC^{\infty}([0,T], \SE)$ a posteriori (like we did in $\Pcurve$).
\end{remark}
\begin{definition}
A (sub-Riemannian) geodesic $\gamma$ of problem $\PMEC$ is a horizontal curve in $\SE$ (i.e.~$\dot{\gamma} \in \Delta$) whose sufficiently short arcs are minimizers of $\PMEC$.
\end{definition}
Next we will address the quotient structure (\ref{OSpace}) and the connection of problem $\Pcurve$ on $\R^3$ to problem $\PMEC$ on $\SE$, and problem $\Pmec$ on $\R^3 \rtimes S^2$.
\begin{remark} Throughout this article we identify $\SOtwo$ with all rotations in $\SO$ about the reference axis, which we choose to be $\ul{e}_{z}$ (i.e. $\SOtwo \equiv \SOtwo \oplus 1$). Furthermore,
$R_{\ul{n}}$ denotes \emph{\underline{any}} rotation
mapping $\ul{e}_{z}=(0,0,1)^T$ onto $\ul{n} \in S^{2}$, whereas $R_{\ul{a},\phi}$ denotes the counter-clockwise rotation about axis $\ul{a}$ over angle $\phi$.
The group $\SE$ acts transitively on $\R^3 \rtimes S^{2}$ by
\begin{equation} \label{act}
g \odot (\ul{y},\ul{n})=(\ul{x},R) \odot (\ul{y},\ul{n})  =(R\ul{y}+ \ul{x}, R \ul{n}).
\end{equation}
Elements in $\SE$ that map $(\ul{0},\ul{e}_{z})$ to itself are denoted by
\[
h_{\alpha}:=(\ul{0},R_{\ul{e}_{z},\alpha}) \in \{\ul{0}\} \times \SOtwo.
\]
Regarding (\ref{OSpace}) we note $
g_{1} \sim g_{2} \desda
g_{1} \odot (\ul{0},\ul{e}_{z})= g_{2} \odot (\ul{0},\ul{e}_{z}) \desda (g_{1})^{-1} g_{2} \in \{\ul{0}\} \times \SOtwo
$
for all $g_{1}, g_{2} \in \SE$.
For sober notation we write $(\ul{y},\ul{n}) \in \R^3 \rtimes S^2$. This represents the left coset
\begin{eqnarray*} \label{leftcoset}
(\ul{y},\ul{n}):=\{g \in \SE \;|\; g \sim (\ul{y}, R_{\ul{n}}) \} = \{(\ul{y},R) \in \SE \;|\; (\ul{y},R) \odot (\ul{0},\ul{e}_{z}) = (\ul{y}, \ul{n})\}\\
= \{(\ul{y}, R_{\ul{n}} R_{\ul{e}_{z},\alpha}) \in \SE \;|\; 0 \leq \alpha < 2\pi\}.\qquad \left.\right.
\end{eqnarray*}
This is similar to the common identification $S^2 = \SO/\SOtwo$.
\end{remark}
We obtain a well-defined distance on the quotient $\R^{3}\rtimes S^{2}$, recall (\ref{OSpace}) and~(\ref{eq:CCdistLip}), by
\begin{equation} \label{onthequotient}
\begin{array}{ll}
d_{\R^{3}\rtimes S^{2}}((\ul{0},\ul{e}_{z}),(\ul{y}_1,\ul{n}_{1})) &= \min \limits_{\alpha^1,\alpha^{2} \in [0,2\pi)} d(e h_{\alpha^1}, (\ul{y}_{1},R_{\ul{n}_1}) h_{\alpha^{2}}) \\
 &= \min \limits_{\alpha^1,\alpha^{2} \in [0,2\pi)} d(e, h_{\alpha^{1}}^{-1} (\ul{y}_1,R_{\ul{n}_{1}}) h_{\alpha^{2}-\alpha^1} h_{\alpha^{1}}) \\
 &= \min \limits_{\alpha \in [0,2\pi)} d(e, (\ul{y}_{1},R_{\ul{n}_{1}}) h_{\alpha}),
 \end{array}
\end{equation}
where we use both left-invariance and invariance under the specific conjugations
$g \mapsto h_{\alpha}^{-1} g h_{\alpha}$ and where we have set $\alpha=\alpha^{2}-\alpha^{1}$. Hence we get the following definition.
\begin{definition}\label{def:Pmec} Problem $\Pmec$ is defined as follows on $\R^{3}\rtimes S^2$.
Let $(\ul{y}_{1},\ul{n}_{1}) \in \R^{3} \rtimes S^{2}$.
Find
\[
[0,T]\ni t \mapsto (\ul{x}(t),\ul{n}(t))= \gamma(t) \odot (\ul{0},\ul{e}_{z}) \in \R^{3}\rtimes S^{2},
\]
with $\gamma$ a Lipschitzian curve in $\SE$ with velocity $\dot{\gamma} \in \Delta$, 
such that sub-Riemannian length
$
\int_0^T
\sqrt{\left.\mathcal{G}_{\xi}\right|_{\gamma(t)}(\dot{\gamma}(t),\dot{\gamma}(t))} {\rm d}t
$ is minimal under boundary conditions $\gamma(0)=(\ul{0},I)$ and
$\gamma(T)=(\ul{y}_{1},R_{\ul{n}_1}R_{\ul{e}_{z},\alpha})$,
where both $T \geq 0$ and $\alpha \in [0,2\pi)$ are free variables in the optimization process.
\end{definition}
In Section~\ref{ch:PMP} we introduce left-invariant Hamiltonians $\h_1,\ldots,\h_6$ linear on the fibers of cotangent bundle $T^{\ast}(\SE)$ and apply the Pontryagin maximum principle (PMP) to the problem $\PMEC$, where the Hamiltonian $H$ is expressed as
\begin{equation} \label{Hamiltonian}
H(\h) = \frac12\left(\ksi^{-2}\, \h_3^{2} + \h_4^2 + \h_5^2\right).
\end{equation}
To distinguish a momentum covector from its components in dual basis we represent momentum covectors $\h(t) = \sum^6_{i=1}\h_i(t)\omega^i|_{\gamma(t)}$ as row vectors $\bh$
\begin{equation}\label{eq:boldlam}
\bh:=(\bh^{(1)},\bh^{(2)}) \textrm{ with }\bh^{(1)}:=(\h_1,\h_{2},\h_3) \textrm{ and }\bh^{(2)}:=(\h_4,\h_5,\h_6),
\end{equation}
where we split $\bh$ into a spatial part $\bh^{(1)}$  and a rotational part $\bh^{(2)}$.
%

A geodesic of problem $\Pmec$ is a curve $(\ul{x}(\cdot),\ul{n}(\cdot)) \in \R^{3} \rtimes S^2$ whose sufficiently short arcs are minimizers of $\Pmec$.
First we will show that solutions to the problem $\Pmec$ in the quotient $\R^3 \rtimes \mathrm{S}^2$ are geodesics in problem $\PMEC$ in $\SE$ iff the initial momentum $\h(0)$ is chosen such that $\h_6=0$. 

Later on, for sub-Riemannian geodesics whose spatial projections
do not exhibit cusps (see the green curves in Fig.~\ref{fig:GeosesicsWithCusps}), we shall rely on the spatial arclength parameter $s$ as this parametrization produces much simpler formulas. To distinguish between derivatives we write
\[\dot{\gamma}(t):=\frac{d}{dt} \gamma(t) \textrm{ and } \gamma'(s):= \frac{d}{ds} \gamma(s).\]

The next theorem provides us the terminal conditions of interest, the appropriate choice of representant in each equivalent class, i.e. the $\alpha$ that minimizes (\ref{onthequotient}). In fact this sets the choice of $(\ul{y}_{1},R_{\ul{n}_{1}}) \in \SE$ in $\PMEC$ such that the spatial projection $\ul{x}^{*}(\cdot)$ of the optimizer of $\gamma^{*}(\cdot)=(\ul{x}^{*}(\cdot), R^{*}(\cdot))$ in $\PMEC$ coincides with the optimizer of $\mathbf{P_{curve}}$. The proof relies on notations and results in Subsection~\ref{ch:PMP}, but we formulate the theorem here as it fully explains the connection between problem $\Pcurve$ on $\R^{3}$, problem $\PMEC$ on $\SE$ and problem $\Pmec$ on $\R^{3}\rtimes S^{2}$. At this point the reader is advised to skip the proof and return to it after Subsection~\ref{ch:PMP}.

\begin{theorem} \label{lemma:One}
If the terminal point $g_{1}=(\ul{x}_1,R_{1}) \in \SE$ is chosen such that a corresponding minimizer $\gamma^{*}$ of $\PMEC$ satisfies\footnote{In Section~\ref{sec:PMEC} we will introduce control variables $u^3$, $u^4$, $u^5$ and will formulate the problem $\PMEC$ as an optimal control problem in $\SE$.
Moreover, it will follow from the Hamiltonian system that $\lambda_6$ is constant along extremals. 
}
\begin{equation} \label{ene}
u^3(t) := \langle\omega^3|_{\gamma^*(t)}, \dot{\gamma}^*(t) \rangle > 0, \textrm{ for all }t \in (0,T),
\end{equation}
then the minimizer $\gamma^{*}$ can be parameterized by spatial arclength $s$, and its spatial projection does not exhibit a cusp.
If moreover $g_{1}$  is chosen such that $\gamma^{*}$ has vanishing momentum component
\begin{equation} \label{tweee}
\lambda_{6}(t)=\lambda_{6}(0)=0, \textrm{ for all }t \in (0,T),
\end{equation}
then this yields the required minimum choice of $\alpha$ in (\ref{onthequotient}), and
the minimizer $\gamma^*(t)$ of $\PMEC$ provides the minimizer
$(\ul{x}^*(t),\ul{n}^*(t))=\gamma^{*}(t) \odot (\ul{0},\ul{e}_{z})$ of problem $\Pmec$.

Under these two requirements (\ref{ene}) and (\ref{tweee}) the spatial projection $\ul{x}^{*}(\cdot)$ of the geodesic $\gamma^{*}(\cdot)=(\ul{x}^{*}(\cdot), R^{*}(\cdot))$ of problem $\PMEC$ coincides with a stationary curve of problem $\Pcurve$.
\end{theorem}
\begin{proof} The proof can be found in Appendix~\ref{app:A}. It relies on notation and results
of Section~\ref{sec:PMEC}. See Figure~\ref{fig:lemma} for an illustration of an explicit case. \end{proof}
\begin{definition}\label{def:expPMEC}
Let
$\widetilde{Exp}_e: \{(\h(0),t) \in T^{*}_{e}(\SE) \times \R^+ \;|\; H(\h(0))=\frac{1}{2}\} \to \SE$ denote the exponential map of
$\PMEC$. This exponential map is the solution operator that solves the Hamiltonian system of PMP, with the Hamiltonian $H(\h(0))$ given by (\ref{Hamiltonian}), departing from $e=(\ul{0},I)$ and thereby it maps initial momentum $\h(0)$ and sub-Riemannian arclength $t$ onto the end-point $(\ul{x}(t),R(t)) \in \SE$ of the corresponding geodesic of problem $\PMEC$.
\end{definition}
For the required case $\h_6=0$ we will derive the operator
$\widetilde{Exp}_e((\h_{1}(0),\h_{2}(0),\h_{3}(0),\h_{4}(0),
\h_{5}(0),0),\; t)$ explicitly in Section~\ref{sec:Pmecs}.
\begin{definition}\label{def:R}
Let $\gothic{R}$ denote the set of boundary points $g_{1}=(\ul{x}_1,R_{1}) \in \SE$ 
reachable by geodesics $\gamma^*$ of problem $\PMEC$ satisfying the two requirements (\ref{ene}) and (\ref{tweee}).
Define $\mathcal{R}=\{(\ul{x}_{1},R_{1}\ul{e}_{z}) \; |\;
 (\ul{x}_{1},R_{1}) \in \gothic{R}\} \subset \R^{3} \rtimes S^{2}$.
\end{definition}
Next we formally define the exponential map of the problem $\Pcurve$, where we rely on the action of $\SE$ onto $\R^{3} \rtimes S^2$, recall (\ref{act}).
\begin{definition}\label{def:exp}
The exponential map $Exp: \mathcal{D}_{0} \to \R^{3} \rtimes S^{2}$ of $\Pcurve$ is defined by:
\[
\begin{array}{l}
Exp(\h(0),L) :=\widetilde{Exp}_{e}(\h(0),T(L)) \odot (\ul{0},\ul{e}_{z})
=\gamma^{*}(T(L)) \odot (\ul{0},\ul{e}_{z})=(\ul{x}^{*}(L),\frac{d}{ds}\ul{x}^{*}(L)),
\end{array}
\]
where $\ul{x}^{*}$ denotes the minimizer of $\Pcurve$, $\gamma^{*}$ denotes a minimizer of $\PMEC$, and with domain
\begin{equation} \label{domain}
\begin{array}{l}
\mathcal{D}_{0}=\{(\h(0),L) \in \mathcal{D} \;| \h_{6}(0)=0\}, \textrm{ with } \\
\mathcal{D}=\{(\h(0),L) \in T^{*}_{e}(\SE) \times \R^+ \;|\; H(\h(0))=\frac{1}{2}, \ L\leq \smax(\h(0)), \ \bh^{(1)}(0) \neq \ul{0}\},
\end{array}
\end{equation}
 where $\smax(\h(0)) \in \R \cup \{\infty\}$ is the maximal length of the spatial projection $\bx^*$ of the sub-Riemannian geodesic $\gamma^*$ to the first cusp.
\end{definition}
We exclude $\bh^{(1)}(0) = \ul{0}$ from the domain $\mathcal{D}$ to avoid pure rotations which are not solutions to the problem $\Pcurve$. This is also done in the $\SEtwo$ case~\cite[Remark 5.5]{DuitsSE2}. By Theorem~\ref{lemma:One} we have the following result.
\begin{corollary} \label{cor:main}
The set $\mathcal{R}$ equals the range of the exponential map of $\Pcurve$:
\begin{equation}\label{eq:rangePcurveSE3}
\mathcal{R}= \{ Exp(\lambda(0),L) \; |\; (\lambda(0),L) \in \mathcal{D}_{0}\} \subset \R^{3} \rtimes S^2,
\end{equation}
 i.e. it coincides with all possible end conditions $(\ul{x}(L),\ul{x}'(L))$ of geodesics of $\Pcurve$.
\end{corollary}
\begin{definition} \label{DEF:AM}
An end condition $(\ul{x}_{1},\ul{n}_{1}) \in \R^{3} \rtimes S^2$ is called admissible if $(\ul{x}_{1},\ul{n}_{1}) \in \mathcal{R}$.
\end{definition}
\begin{figure}
\centerline{
\includegraphics[width=0.5\hsize]{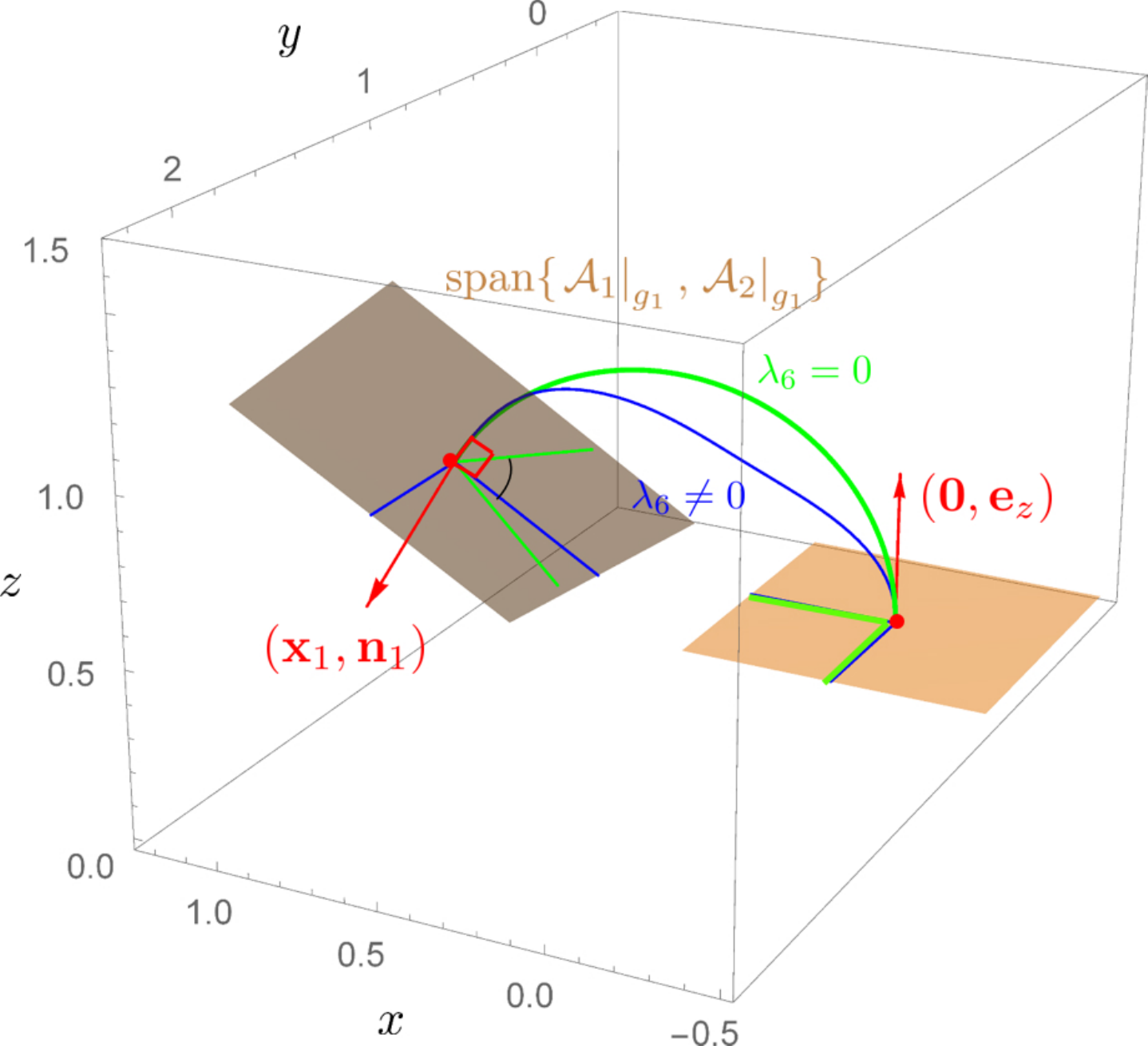}
}
\caption{End point $(\ul{x}_{1},\ul{n}_{1})=((0.53,1.8,1.1), (0.031,0.86,-0.51))$ in $\Pcurve$ gives rise to many possible end conditions
$(\ul{x},R_{\ul{n}_{1}}) \in \SE$ in $\PMEC$. By Theorem~\ref{lemma:One}, the minimizer in (\ref{onthequotient}) is found by setting $\h_{6}(0)=0$.
Here the spatial projection of the minimizing geodesic is depicted in green, and the spatial projection of a geodesic with $\h_6 \neq 0$ in blue.
In order to show the choice of rotation $R_{\ul{n}_{1}} \in \SO$, s.t. $R_{\ul{n}_{1}}\ul{e}_{z}=\ul{n}_{1}$, we depict the spatial frame $\{\left.\mathcal{A}_{1}\right|_{g_1},\left.\mathcal{A}_{2}\right|_{g_1},\left.\mathcal{A}_{3}\right|_{g_1}\}$ at the end points $g_{1}=(\ul{x}_{1},R_{\ul{n}_1})$ of both the blue and green geodesic. For the minimizing geodesic we have $\lambda_6(0)=0$ and $T=d_{\mathbb{R}^{3}\rtimes S^{2}}((\ul{0},\ul{e}_{z}),(\ul{y}_1,\ul{n}_1))=4$, for the other geodesic we have $\h_6(0)=2$ and indeed $T=4.65> 4$.
\label{fig:lemma}}
\end{figure}
\section{$\PMEC$: Sub-Riemannian Problem on $\SE$}\label{sec:PMEC}
In this section we study the problem $\PMEC$. The optimal control theoretical formulation of this sub-Riemannian problem is to find a Lipschitzian curve $\gamma : [0,T]\rightarrow \SE$, with boundary conditions $\gamma(0)=e:=(\bzero,I)$ and $\gamma(T)=(\bx_1,R_1) \in \SE$, such that
\begin{equation}\label{eq:PMECStatControl}
\begin{array}{l}
\int \limits_{0}^{T}\!{\sqrt{\mathcal{G}_{\xi}|_{\gamma (t)}(\dot{\gamma}(t),\dot{\gamma}(t))}\,\dif t} =  \int \limits_0^T\!\sqrt{\xi^2(u^3(t))^2+(u^4(t))^2+(u^5(t))^2}\,\dif t
\to \min \mbox{ (with free $T$)}, \\[8pt]
\textrm{with }
\dgamma (t)  = \sum \limits_{i=3}^5u^i(t)\cA_i|_{\gamma(t)} = \sum \limits_{i=3}^5\langle\omega^i|_{\gamma(t)}, \dgamma(t)\rangle\cA_i|_{\gamma(t)},
\end{array}
\end{equation}
where the control variables $u^i\in \mathbb{L}_1([0,T])$ for $i=3,4,5$.
In particular, we only consider the curves for which the absolute curvature of the spatial projections is in $\mathbb{L}_1([0,T])$. 
The control variables are contravariant components of the velocity vector, so we index them with upper indices.
\begin{remark}
The problem $\PMEC$ given by~(\ref{eq:PMECStatControl}) can be seen as a motion planing problem for a spacecraft, that can move forward/backward (in direction $\cA_3$) and rotate about axis $u^4 \cA_1 + u^5 \cA_2$.
The control $u^3$ determines spacial velocity (the sign of $u^3$ determines forward/backward propagation), while the controls $u^4$ and $u^5$ determine the angular velocity, cf. Fig.~\ref{movingframe}.
\end{remark}
The existence of minimizers for the problem $\PMEC$ is guaranteed by the theorems by Chow-Rashevskii and Filippov on sub-Riemannian structures~\cite{notes}. 
\begin{remark}\label{abnormal}
About smoothness of minimizers of $\PMEC$.
\begin{itemize}
\item We have that for $\PMEC$, there are no abnormal extremals. It follows from the fact that any sub-Riemannian manifold with a $2$-generating distribution does not allow abnormal extremals (see Chapter 20.5.1 in \cite{notes}). This is the case here as we have for $\Delta := \{\cA_3,\cA_4,\cA_5\}$, $\dim(\,[\Delta,\Delta])=\dim(\spann\{\cA_1,\cA_2,\cA_3,\cA_4,\cA_5,\cA_6\})=\dim\,(\SE)=6$.
\item Due to the non-existence of abnormal extremals, the geodesics are always analytic \cite{notes} and so are the extremal controls. A priori, in the application of the Pontryagin Maximum Principle (PMP) \cite{notes,Vinter2010}, one cannot restrict to smooth controls where one relies on $\mathbb{L}_{\infty}$-controls for $\PMEC$ (and $\mathbb{L}_{1}$-controls for $\Pcurve$, similar to the $\SEtwo$-case \cite[ch:5.1, App.B]{Boscain2013a}). In retrospect, however, the minimizers
    are analytic, and one can write the sub-Riemannian distance from $g$ to $e=(\ul{0},I)$ on $(\SE,\Delta,\mathcal{G}_{\xi})$ as
    \begin{equation} \label{distance}
    d(g,e)= \min \limits_{
    \begin{array}{c}
    \gamma \in C^{\infty}([0,T], \SE), T\geq 0, \\
    \dot{\gamma} \in \Delta, \
    \gamma(0)=e, \gamma(T)=g
    \end{array}} \int \limits_{0}^T
    \sqrt{\left. \mathcal{G}_{\xi}\right|_{\gamma(t)}(\dot{\gamma}(t),\dot{\gamma}(t))}\, {\rm d}t.
    \end{equation}
\end{itemize}
\end{remark}
Next, in the application of PMP to the problem $\PMEC$, we rely on the sub-Riemannian arclength parameter $t$, which is well-defined for all SR-geodesics in $(\SE,\Delta, \mathcal{G}_{\ksi})$. Before a cusp occurs, recall Fig.~\ref{fig:GeosesicsWithCusps}, one can also use spatial arclength parameterization $s$, related by $t(s) = \int_0^s \sqrt{\ksi^2+\kappa^2(\sigma)} \; \dif \sigma$. The formal defenition of a cusp time is given bellow.
\begin{definition}\label{def:cusptime}
A cusp time is a time $0<\tcusp<T$ when the third control component $u^3(t)$  (responsible for spatial propagation in $\PMEC$) vanishes, $u^{3}(\tcusp)=0$ and moreover $\dot{u}^3(\tcusp) \neq 0$.
\end{definition}
For illustrations of cusps see Fig.~\ref{fig:GeosesicsWithCusps} and for planar sub-Riemannian geodesics see \cite[Fig.2]{DuitsSE2}, \cite{Boscain}.
\subsection{Application of Pontryagin Maximum Principle}\label{ch:PMP}
A first order necessary optimality condition is given by Pontryagin maximum principle (PMP)~\cite{notes,Pontryagin}. Note that PMP gives a necessary but not a sufficient condition of optimality. Geodesics of $\PMEC$ loses local optimality after the first conjugate point 
and global optimality after the first Maxwell point (see ~\cite{notes}).

The Cauchy-Schwarz inequality implies that the minimization problem for the sub-Riemannian length functional is equivalent to the minimization problem for the action functional (see~\cite{notes})
\begin{equation}\label{eq:actionint}
J = \frac12 \int_0^T (\ksi^2 (u^3)^2 +(u^4)^2 + (u^5)^2) \; \dif t \to \min, \quad \textrm{ with fixed $T>0$.}
\end{equation}
Next we apply PMP to problem $\PMEC$ using the equivalent action functional, i.e. to the following optimal control problem:
\begin{equation}\label{eq:PMECCS}
\dot{\gamma}(t) = \sum_{i = 3}^{5}u^i(t)\mathcal{A}_i|_{\gamma(t)},
\qquad J \to \min, \qquad
\gamma(0) = e, \; \gamma(T) = g_1 \in \SE.
\end{equation}

Denote $q = (x,y,z,\tga,\tbe,\tal) \in \R^3\times (-\pi,\pi] \times[-\frac{\pi }{2}, \frac{\pi }{2}]\times(-\pi, \pi]$ which is identified with an element from $\SE$ via (\ref{eq:Rparam}). The natural momentum coordinates $\{\h_i\}$ for left-invariant sub-Riemannian problems come along with the left-invariant Hamiltonians
$h_{i}:T^{*}(\SE) \to \R$ (see~\cite{notes}),
given by
\begin{equation*}
 h_i(q,\lambda) = \langle \lambda(q) , \AA_i|_{q} \rangle=\lambda_{i}(q), \qquad i = 1,\dots 6,
\end{equation*}
where $\lambda(q) = p_1(q) {\rm d}x|_{q} +p_2(q) {\rm d}y|_{q} +p_3(q) {\rm d}z|_{q} +p_4(q) {\rm d}\tga|_{q} + p_5(q)
 {\rm d}\tbe|_{q} + p_6(q) {\rm d}\tal|_{q}= \sum_{i=1}^6\lambda_{i}(q) \, \omega^{i}|_{q} \in T_q^{\ast}(\SE)$ denotes a momentum covector expressed in respectively the fixed and the moving dual frame.

Now we apply PMP. By Remark~\ref{abnormal} we only need to consider the normal case. The control-dependent Hamiltonian reads
$H_u = u^3 \h_3 + u^4 \h_4 + u^5 \h_5 - \frac12 \left( \ksi^2 (u^3)^2 +(u^4)^2 + (u^5)^2\right)$. Optimization over all controls produces the (maximized) Hamiltonian
\[H = \max_{u\in\R^3}H_u  = \frac{1}{2}\left(\xi^{-2}\h_3^2 + \h_4^2 + \h_5^2\right),\]
and gives expression for the extremal controls
\[u^3 = \frac{\h_3}{\ksi^2}, \quad u^4 = \h_4, \quad u^5 = \h_5.\]
By virtue of the Lie brackets (see Table \ref{tabcomA}), we have the Poisson brackets
\begin{eqnarray*}
&&\{H,\h_1\} = -\h_3 \h_5, \quad \{H,\h_2\} = \h_3 \h_4,\quad\{H,\h_3\} = \h_1 \h_5 - \h_2 \h_4, \\
&& \{H,\h_4\} = \frac{\h_2 \h_3}{\ksi^2} - \h_5 \h_6,\quad \{H,\h_5\} = \h_4 \h_6 - \frac{\h_1 \h_3}{\ksi^2}, \quad \{H,\h_6\} = 0.
\end{eqnarray*}
Thus the Hamiltonian system of PMP reads as follows:
\begin{equation} \label{eq:hamsys}
\begin{array}{ll}
\begin{cases}
\dot{\h}_1 = -\h_3 \h_5,\\
\dot{\h}_2 = \h_3 \h_4,\\
\dot{\h}_3 = \h_1 \h_5 - \h_2 \h_4,\\
\dot{\h}_4 = \frac{\h_2 \h_3}{\ksi^2} - \h_5 \h_6,\\
\dot{\h}_5 = \h_4 \h_6 - \frac{\h_1 \h_3}{\ksi^2},\\
\dot{\h}_6 = 0,
\end{cases} 
& \begin{cases}
\dot{x} = \frac{\h_3}{\ksi^2} \sin\tbe, \\
\dot{y} = -\frac{\h_3}{\ksi^2}\cos\tbe \sin\tga, \\
\dot{z} = \frac{\h_3}{\ksi^2}\cos\tbe \cos\tga,\\
\dot{\tga} = \sec\tbe (\h_4  \cos\tal - \h_5 \sin\tal), \\
\dot{\tbe} = \h_4 \sin\tal +\h_5\cos\tal, \\
\dot{\tal} = -(\h_4 \cos\tal - \h_5 \sin\tal) \tan\tbe,
\end{cases} \\ 
 \text{--- vertical part (for adjoint variables),}
&  \text{--- horizontal part (for state variables).}
\end{array}
\end{equation}
These equations only hold outside the singularities of the angles chart $\{\tal,\tbe,\tga\}$. In particular if $\tbe \in \{-\frac{\pi}{2},\frac{\pi}{2}\}$
we can rely on standard Euler angles $R=R_{\ul{e}_{z},\bar{\gamma}}R_{\ul{e}_{y},\bar{\beta}} R_{\ul{e}_{z},\bar{\alpha}}$ and in the vicinity
of those singularities, the final 3 equations of the
horizontal part need to be replaced by  
\[
\begin{array}{l}
\dot{\bar{\gamma}}= -\frac{\cos \bar{\alpha}}{\sin \bar{\beta}} \h_4 + \frac{\sin \bar{\alpha}}{\sin \bar{\beta}} \h_5, \ \dot{\bar{\beta}}= \h_4 \sin \bar{\alpha} + \h_5 \cos \bar{\alpha}, \ \dot{\bar{\alpha}}= (\h_4 \cos \bar{\alpha} - \h_5 \sin \bar{\alpha})\cot \bar{\beta}.
\end{array}
\]
The sub-Riemannian geodesics are solutions to the Hamiltonian system. Finding a parameterization of the sub-Riemannian geodesics is a nontrivial problem.
In order to guarantee that such a parametrization exists, we first prove Liouville integrability of the system.
Here we follow the same approach as in~\cite{Mashtakov_Sachkov}.

The next remark shows that we can restrict ourselves to the case $\ksi = 1$.
\begin{remark}\label{rem:homothety}
The Hamiltonian system of PMP (\ref{eq:hamsys}) reveals the scaling homothety. 
The case $\ksi \neq 1$ is obtained from the case $\ksi=1$ by
\[
\blambda \mapsto \blambda (\Lambda_{\ksi}^{-1} \oplus I_{3})\textrm{ and } (\ul{x},R) \mapsto (\Lambda_{\ksi} \ul{x},R),
\]
with $\Lambda_{\ksi}=\diag\{\ksi,\ksi,\ksi\}$. So in order to obtain solution $\ul{x}^{*}$ of $\Pcurve$ with boundary conditions
$(\ul{x}(0),\ul{x}'(0))=(\ul{0},\ul{e}_{z})$  and $(\ul{x}(L),\ul{x}'(L))=(\ul{x}_1,\ul{n}_{1})$
for $\ksi>0$, we first solve $\Pcurve$ with boundary conditions
$(\ul{x}(0),\ul{x}'(0))=(\ul{0},\ul{e}_{z})$  and $(\ul{x}(L),\ul{x}'(L))=(\ksi\ul{x}_1,\ul{n}_{1})$
for $\ksi=1$ in dynamics~(\ref{eq:hamsys}). Then the optimal curve $\ul{x}(\cdot)$ needs to be scaled back
$\ul{x}^{*}(s)=\ksi^{-1} \ul{x}(s)$. The homothety boils down to making the problem dimensionless, as the physical dimension of $\ksi^{-1}$
is spatial length.
\end{remark}
\begin{remark}
All prerequisites for the proof of Theorem~\ref{lemma:One} are now given. See Appendix~\ref{app:A}.
\end{remark}
\subsection{Liouville Integrability\label{subsec:Integrability}}
To prove the Liouville integrability of the Hamiltonian system~(\ref{eq:hamsys}), one should construct a complete
system of first integrals, i.e. indicate six first integrals in involution (w.r.t. Poisson brackets)
and functionally independent on an open dense domain in $T^{\ast}(\SE)$~\cite[p. 107]{arnold_kozlov}.

It is well known that the Hamiltonian
$H = \frac{1}{2}\left(\h_3^2 + \h_4^2 + \h_5^2\right)$
is a first integral of the Hamiltonian system.
From the vertical part of~(\ref{eq:hamsys}) we can immediately see one more first integral $\h_6$, which is functionally independent from $H$.
Since $\{H, \h_6\} =0$, we see that the integrals $H$ and $\h_6$ are in involution.
The Hamiltonian system directly reveals the first integral $W = -\h_1 \h_4 - \h_2 \h_5 - \h_3 \h_6$. This integral is a Casimir function, i.e. $\{W, \h_i\} = 0, i = 1, ..., 6$. Casimir functions are functions on the dual space of the Lie algebra commuting in the sense of Poisson brackets with all left-invariant Hamiltonians. They are universal conservation laws on the Lie group. Connected joint level surfaces of all Casimir functions are coadjoint orbits (see~\cite[Prop.~7.7]{Gengoux}). Since these orbits are always even-dimensional (they are symplectic manifolds) the difference between the dimension of the Lie group and the number of Casimir functions is even.
Casimir functions are found by solving $\{C, \h_i\} = 0$.
For polynomial functions $C$, we can write $C$ with undetermined coefficients as a polynomial of degree $1, 2, 3$, and solve the resulting system of equations algebraically. The second Casimir function in $\SE$ is given by $\gc^2 = \h_1^2+\h_2^2+\h_3^2$. For details on the Casimir functions see the work of A.A.~Kirillov~\cite{kirillov} or the book of V.~Jurdjevic~\cite{jurdjevic}.

To construct the complete system of first integrals we consider integrals $H$, $\h_6$ and $W$ and
find 3 more first integrals. Then we show that all 6 integrals are functionally independent on an open dense domain in $T^{\ast}(\SE)$ and are in involution.
\subsubsection{Right-Invariant Hamiltonians}
For any right-invariant vector field $\mathcal{B} \in \Vec(\SE)$, one can define the corresponding Hamiltonian $\g(q,\lambda) = \langle \lambda(q), \mathcal{B}|_{q}\rangle$,  $\lambda(q) \in T^{\ast}_{q}(\SE)$.
Since right translations commute with the left ones,  left-invariant vector fields commute with   right-invariant ones. Thus left-invariant Hamiltonians Poisson-commute with the right-invariant ones. 

The right-invariant vector fields are given by
\begin{eqnarray*}
\mathcal{B}_{1}  = -\partial_x, \quad
\mathcal{B}_{2}  = -\partial_y, \quad
\mathcal{B}_{3}  = -\partial_z, \quad
\mathcal{B}_{4}  = z \partial_y - y \partial_z -\partial_{\tga},\\
\mathcal{B}_{5}  = -z \partial_x + x \partial_z -\sin\tga \sec\tbe \partial_\tal -\cos\tga \partial_\tbe - \sin\tga \tan\tbe \partial_\tga,\\
\mathcal{B}_{6}  = y \partial_x - x \partial_y -\cos\tga \sec\tbe \partial_\tal -\sin\tga \partial_\tbe + \cos\tga \tan\tbe \partial_\tga.
\end{eqnarray*}

Then the right-invariant Hamiltonians are defined by $\g_{i}(q,\lambda)=\langle \lambda(q),\mathcal{B}_{i}|_{q}\rangle$. Since the table of Poisson brackets between $\g_i$ coincides with the commutator of corresponding vector fields (cf.~\!Table~\ref{tabcomA}), we see that only $\g_1, \g_2$ and $\g_3$ are in involution. Their expression via the left-invariant Hamiltonian $\h_i$ reads as
\begin{equation}
\label{rhos}
\begin{array}{ll}
&\g_{1}  = -\h_1 \cos\tal \cos\tbe+\h_2 \cos\tbe \sin\tal-\h_3 \sin\tbe,\\
&\g_{2}  = -\cos\tga (\h_2 \cos\tal + \h_1 \sin\tal) + (\h_3 \cos\tbe + (-\h_1 \cos\tal + \h_2 \sin\tal) \sin\tbe) \sin\tga ,\\
&\g_{3}  = -\h_3 \cos\tbe \cos\tga + \cos\tga (\h_1 \cos\tal - \h_2 \sin\tal) \sin\tbe - (\h_2 \cos\tal +
    \h_1 \sin\tal) \sin\tga. 
\end{array}
\end{equation}
\subsubsection{Independence of Integrals}
In order to study the functional independence of the integrals $H$, $\h_6$, $W$, $\g_1$, $\g_2$, $\g_3$ at a point $(q,\h) \in T^{\ast}( \SE)$, introduce the Jacobian matrix
$$J(q,\h) = (\nabla H \, |\, \nabla \h_6 \,|\, \nabla W \,|\, \nabla \g_1 \,|\, \nabla \g_2 \,|\, \nabla \g_3)^T|_{(q,\h)}.$$
Liouville integrability of the Hamiltonian system follows by the study of the vertical derivatives of the integrals (i.e., the derivatives w.r.t. the variables $\h_i$).  By analyticity, functional independence of the integrals on an open dense domain in $T^{\ast}(\SE)$ follows from linear independence of gradients of the integrals at a single point $(q,\h) \in T^{\ast}(\SE)$. Since we have
\begin{eqnarray*}
& \frac{\partial(\g_1,\g_2,\g_3,W,H,\h_6)}{\partial (\h_1,\h_2,\h_3,\h_4,\h_5,\h_6)}(q,\h) = 
& -\h_2 \h_4 + \h_1 \h_5 \not\equiv 0,
\end{eqnarray*}
the first integrals $\g_1$, $\g_2$, $\g_3$, $I$, $H$, $\h_6$ are functionally independent.
Here we use that $-\h_{2}\h_{4} + \h_{1}\h_{5}=\dot{\h}_{3}= \dot{u}^{3} = \frac{d^2 s}{dt^2} \not\equiv 0$.
So we proved the following theorem.
\begin{theorem}
\label{th:integrability}
The Hamiltonian system~(\ref{eq:hamsys}) is Liouville integrable. Functionally independent first integrals are $\rho_1$, $\rho_2$, $\rho_3$, cf.~\!Eq.\!~(\ref{rhos}), $W=-\!\lambda_{1}\lambda_4\!-\!\lambda_{2}\lambda_{5}\!-\!\lambda_3\lambda_6$,
$H=\frac{1}{2}(\lambda_3^2\!+\!\lambda_{4}^2\! +\!\lambda_{5}^2)$ and $\lambda_6$. \end{theorem}
Explicit integration of~(\ref{eq:hamsys}) in sub-Riemannian arclength parametrization $t$ is a difficult problem. Further in Section~\ref{sec:Pmecs} we show that using spatial arclength  parametrization $s$ leads to relatively simple expressions for sub-Riemannian geodesics whose
spatial projections do not exhibit cusps.
\subsection{The $-$ Cartan Connection $\overline{\nabla}$}\label{ch:Cartan}
In the sub-Riemannian manifold $(\SE, \Delta, \mathcal{G}_{\ksi})$, the directions $\cA_1$, $\cA_2$ and $\cA_6$ are prohibited in the tangent bundle. To get a better grasp on what this means on the manifold level, we consider principal fibre bundles. We use the minus `$-$' Cartan connection~\cite{CS} to connect the Hamiltonian PMP approach to the Lagrangian reduction approach by Bryant and Griffiths \cite{Bryant}. It provides more intuition and is an important tool towards explicit formulas. As is seen by the following theorem, these curves actually describe parallel transport of the momentum covectors w.r.t. Cartan connections. In the original Cartan and Schouten article~\cite{CS}, three Cartan connections are presented, the $+$, the $-$, and $0$ connection. Here we shall rely on the minus $-$ Cartan connection for which the left-invariant vector fields are autoparallel~\cite{Penne}, since our geometrical control problem is expressed in left-invariant vector fields. In order to keep track of correct tensorial computations we deal with the general case $\ksi>0$ in Theorem~\ref{th:2}. However, by Remark~\ref{rem:homothety} only the case $\xi=1$ is considered in the remainder of the article.
\begin{definition} \label{def:CC}
We define connection $\overline{\nabla}$ on the (horizontal) tangent bundle of $(\SE,\Delta,\mathcal{G}_{\ksi})$ by
\begin{align} \label{CC}
\overline{\nabla}_{\dot{\gamma}}\cA &:= \sum_{k=3}^{5}\left((\dot{a}^k) - \sum \limits_{i,j=3}^{5}c^{k}_{i,j} \, (\dot{\gamma}^i) a^{j} \right)\mathcal{A}_{k},
\end{align}
with ${\dot{\gamma}}=\sum \limits_{i=3}^{5}\dot{\gamma}^i \cA_i|_\gamma$, $\cA=\sum \limits_{k=3}^{5}a^k\cA_k$ and Lie algebra structure constants $c^{k}_{i,j}$ given in Table~\ref{tabcomA}.
\end{definition}
It is a partial $-$ Cartan connection that originates from a specific choice of principle fiber bundle. For details see Appendix~\ref{app:Cartan} below.
Another ingredient in the theorem below is the standard exponential map from Lie algebra to Lie group given by $T_{(\ul{0},I)}(\SE) \ni A \to e^A \in \SE$.
\begin{theorem}\label{th:2}
Horizontal exponential curves in $(\SE,\Delta,\mathcal{G}_{\ksi})$,
are given by $t\mapsto g_0\; e^{t\sum\limits_{i=3}^{5} c^i A_i}$,
with $\ksi^2(c^{3})^2+ (c^{4})^2 + (c^{5})^2=1$, $g_0 \in \SE$,
and they are precisely the auto-parallel curves, i.e.
\[
\overline{\nabla}_{\dot{\gamma}}\dot{\gamma}=0.
\]
Along sub-Riemannian geodesics in $(\SE,\Delta,\mathcal{G}_{\ksi})$ one has covariantly constant momentum, i.e.
\begin{equation} \label{imp1}
  \overline{\nabla}_{\dot{\gamma}}^*\lambda:= \sum_{i=1}^{6}
\left(\dot{\h}_{i} + \sum \limits_{j=3}^{5} \sum \limits_{k=1}^{6} c^{k}_{i,j} \h_{k} \,\dot{\gamma}^{j}\right) \omega^{i}=0.
\end{equation}
When setting contravariant components $\h^{i}= \ksi^{i} \h_{i}$, $\ksi^{3}=\ksi^{-2}$, $\ksi^{4}= \ksi^{5}=1$, and $P_{\Delta^{\ast}}\h=\sum \limits_{i=3}^5\h_{i} \omega^{i}$
then the Hamiltonian system of PMP~(\ref{eq:hamsys}) can be written as:
\[
\begin{array}{c c c}
\forall_{i\in\{1, \ldots, 6\}}:
\dot{\h}_{i} + \sum \limits_{j=3}^{5} \sum \limits_{k=1}^{6} c^{k}_{i,j} \h_{k} \, \h^j = 0 & \textrm{ and } & \forall_{i\in\{3,4, 5\}}: \dot{\gamma}^{i}=\h^{i} \ \ 
 \\
\mathrm{(vertical~~part)} & & \mathrm{(horizontal~~part),}
\end{array}
\]
which is equivalent to $\overline{\nabla}_{\dot{\gamma}}^*\h= 0$  and $\mathcal{G}_{\ksi}^{-1}P_{\Delta^{\ast}} \h = \dot{\gamma}$.
\end{theorem}
\begin{proof}
Define $\dot{\gamma}^{i}:= \langle \left.\omega^{i}\right|_{\gamma}, \dot{\gamma} \rangle$.
Note that $\dot{\gamma}^i = u^i$. Here we write $\dot{\gamma}^i$ (instead of $u^i$) to stress the curve dependence of the control components.
Then following the same approach as done in \cite{DuitsAMS2}, \cite[App.C]{DuitsSE2} (for
the case $\SEtwo$), the Cartan connection on the tangent bundle is given by Eq.~(\ref{CC}). From Eq.~(\ref{CC}) we see that
the auto-parallel curves are horizontal exponential curves:
\[ \bar{\nabla}_{\dot{\gamma}} \dot{\gamma}=0 \desda \forall_{i \in \{3,4,5\}} \ddot{\gamma}^{i}=0  \desda \forall_{i \in \{3,4,5\}} \dot{\gamma}^{i}=u^{i}=c^{i}  \textrm{ for some constants } c^{i} \desda \gamma(t)= \gamma(0)\, e^{t\sum\limits_{i=3}^{5} c^i A_i},\]
and to ensure $t$ to be the sub-Riemannian arclength parameter we must have $\ksi^2(c^{3})^2+(c^{4})^2 + (c^{5})^2=1$.
Now the partial (`right' or $-$) Cartan connection $\overline{\nabla}$ on the tangent bundle  naturally imposes the partial Cartan connection $\overline{\nabla}^*$ on the cotangent bundle, as follows:
\begin{equation} \label{covcot}
\overline{\nabla}_{\dot{\gamma}}^*\sum_{i=1}^{6}
\h_{i}\omega^i|_{\gamma} =
\sum \limits_{i=1}^{6} \left(\dot{\h}_{i} +
\sum \limits_{j=3}^{5} \sum\limits_{k=1}^{6} c^k_{i,j}\h_k \dot{\gamma}^j\right) \left.\omega^{i}\right|_{\gamma}\ ,
\end{equation}
which follows from Eq.~(\ref{CC}) and
$
\dif\langle\omega^k|_{\gamma},\cA_j|_{\gamma}\rangle =\langle\overline{\nabla}^{*}_{\dot{\gamma}}\omega^k|_{\gamma},\cA_j|_{\gamma}\rangle + \langle\omega^k|_{\gamma},\overline{\nabla}_{\dot{\gamma}}\cA_j|_{\gamma}\rangle =0
$, and $c^{k}_{i,j}=-c_{j,i}^k$.
Now, by the horizontal part of PMP, we have
$
\dot{\gamma}^{i}=\h^{i}\textrm{ for all }i \in \{3,4,5\},
$
so that the result follows by substituting this equality
into Eq.~(\ref{covcot}).
\end{proof}

Next we switch to spatial arclength parameter $s$, as this is convenient. Recall $\frac{ d}{d s}$ is denoted by a prime.
Also recall that $s$-parametrization is well defined until the spatial projection of a sub-Riemannian geodesic exhibits a cusp (recall Definition~\ref{def:cusptime}).

Denote by $\smax$ the minimum positive value of the parameter $s$ where such a cusp appears. Explicit expression for $\smax$ in terms of the initial momentum $\h(0)$ will follow (see Eq.~\!(\ref{smaxfrml})).
Next to find the sub-Riemannian geodesics we integrate the equation in Theorem~\ref{th:2} via a suitable matrix representation
visible in the Cartan-matrix of the Cartan connection.
Such a group representation $m: \SE \to \Aut(\R^{6})$ is given by
\begin{equation} \label{M}
m(\ul{x},R):=
\left(
\begin{array}{cc}
R & \sigma_{\ul{x}}R \\
0 & R
\end{array}
\right),
\end{equation}
where $\sigma_{\ul{x}}=\sum \limits_{i=1}^{3} x^{i} A_{3+i} \in \textrm{so}(3)$, so that $\sigma_\ul{x}\ul{y} = \ul{x}\times\ul{y}$, with
$\ul{x}=\sum \limits_{i=1}^{3} x^{i} \ul{e}_{i}$ and $A_{3+i} \in \textrm{so}(3)$ given by
$$A_4 = \left(\begin{array}{ccc}
0 & 0 & 0 \\
0 & 0 & -1 \\
0 & 1 & 0
\end{array}\right), \quad
A_5 = \left(\begin{array}{ccc}
0 & 0 & 1 \\
0 & 0 & 0 \\
-1 & 0 & 0
\end{array}\right),\quad
A_6 = \left(\begin{array}{ccc}
0 & -1 & 0 \\
1 & 0 & 0 \\
0 & 0 & 0
\end{array}\right).$$
Here,
we have $\sigma_{R\ul{x}}=R \sigma_{\ul{x}} R^{-1}$ and thereby $m(g_{1}g_{2})=m(g_{1}) m(g_{2})$ for all $g_{1},g_{2} \in \SE$. Then
\[
\begin{array}{lll}
{\rm d}\bh &= \bh
m(\gamma^{-1}) {\rm d} m(\gamma)= \bh \;
\left(
\begin{array}{cc}
R^{-1}{\rm d}R & \sigma_{R^{-1} {\rm d}\ul{x}} \\
0 & R^{-1}{\rm d}R
\end{array}
\right) & = \bh \;
\left(
\begin{array}{cc}
\sigma_{(\omega^{4},\ldots,\omega^{6})^T} & \sigma_{(\omega^{1},\ldots,\omega^{3})^T} \\
0 & \sigma_{(\omega^{4},\ldots,\omega^{6})^T}
\end{array}
\right),
\end{array}
\]
with short notation $\omega^{j}=\left.\omega^{j}\right|_{\gamma}$, $\bh=\left.\bh\right|_{\gamma}$, ${\rm d}\bh=\left.{\rm d}\bh\right|_{\gamma}$,
and where we represent the covector $\h=\sum_{i=1}^{6}\h_{i}\left.\omega^{i}\right|_{\gamma}$ by a row-vector $\bh=(\h_{1}, \ldots,\h_{6})$.
So we see that~(\ref{M}) naturally appears in~(\ref{imp1}):
$
\overline{\nabla}_{\dot{\gamma}}^*\h = 0 \desda \frac{{\rm d} \bh}{\dif t} - \bh m(\gamma^{-1}) \frac{{\rm d} m(\gamma)}{{\rm d} t} = 0.
$
\begin{theorem}\label{th:3}
Let $m: \SE \to \Aut(\R^{6})$ denote the matrix group representation (\ref{M}) 
s.t.
\begin{equation} \label{imp2}
\left.{\rm d}\bh\right|_{\gamma} =  \left.\bh \right|_{\gamma} \; m(\gamma^{-1}) {\rm d} m(\gamma).
\end{equation}
Then along the sub-Riemannian geodesics in $(\SE,\Delta,\mathcal{G}_{1})$ the following relation holds:
\[
\bh(s) m(\gamma(s))^{-1}= \bh(0) \, m(\gamma(0))^{-1} = \bh(0).
\]
\end{theorem}
\begin{proof}
Note that $\overline{\nabla}^*_{\gamma'(s)} \left.\h\right|_{\gamma(s)} =0$ iff
$
\frac{{\rm d}}{{\rm d} s}
\left.\bh(s)\right|_{\gamma(s)} -
\left.\bh(s) \right|_{\gamma(s)} \; m((\gamma(s))^{-1}) \frac{{\rm d}}{{\rm d} s}m(\gamma(s))=0
$
for all $0\leq s\leq \smax(\h(0))$. The rest follows by
\[
\frac{{\rm d}}{{\rm d}s}(\bh(s) (g(s))^{-1})= -\bh(s) (g(s))^{-1} g'(s) (g(s))^{-1} + \bh'(s) (g(s))^{-1}=0\ ,
\]
with $g(s)= m(\gamma(s))$. Multiplication with $g(s)$ from the right yields the result.
\end{proof}
For further details see App.~\ref{app:Cartan}. These details are not necessary for the remainder of the article, where we only rely on (\ref{CC}), (\ref{imp1}), (\ref{M}) and (\ref{imp2}).

Let us recall the first integrals of the Hamiltonian system and the coadjoint orbits, that we use next for the derivation of explicit formulas for the geodesics.
\begin{lemma}\label{lemme:coadjorb}
Co-adjoint orbits of $\h(0)$ are given by (see~\cite[p.~474]{Marsden_Ratiu} and Section~\ref{subsec:Integrability})
$$\{\h \in T^{\ast}(\SE)\; |\; C_1(\h) = C_1(\h(0)) = \gc^2, C_2(\h) = C_2(\h(0)) = W\},$$
where $C_i$ are the Casimir functions
\begin{equation}\label{eq:Wgc}
C_1(\h) = \h_1^2 + \h_2^2 +\h_3^2 = \gc^2, \quad C_2(\h) = -\h_1 \h_4 - \h_2 \h_5 - \h_3 \h_6 = W.
\end{equation}
\end{lemma}
\begin{corollary} \label{cor:new}
On each co-adjoint orbit we can choose the nice representative $\bh(0) = (\gc,0,0,-\frac{W}{\gc},0,0)$, solve the Exponential map for this representative and obtain the general solution by left-invariance.
More precisely, by Theorem~\ref{th:3}, we first find the geodesic $\tilde{\gamma}$ with
\begin{equation}\label{eq:representative}
\bh(s) = \bh(0)\, m(\gamma(s)) = \bh(0)\, m(\tilde{\gamma}(0))^{-1}\, m(\tilde{\gamma}(s)) = (\gothic{c},0,0,-\frac{W}{\gothic{c}},0,0)\, m(\tilde{\gamma}(s))
\end{equation}
and then we obtain $\gamma$ via
 $\gamma(s) = \tilde{\gamma}^{-1}(0) \tilde{\gamma}(s)$.
\end{corollary}

\section{Sub-Riemannian Geodesics in $\R^{3}\rtimes S^{2}$ with Cuspless Projections}\label{sec:Pmecs}
%
%
In this section, we derive sub-Riemannian geodesics with cuspless spatial projection in the quotient $\R^3 \rtimes S^2$, and we study geometrical properties of the geodesics, such as planarity conditions and bounds on the torsion. By Remark~\ref{rem:homothety} we set $\ksi=1$.

Recall that from Theorem~\ref{lemma:One}, application of PMP to $\Pmec$ follows from PMP for the problem $\PMEC$ putting initial momentum $\h_6 = 0$. This yields the following ODE for the horizontal part:
$$\dgamma = \h_3\cA_3|_{\gamma} + \h_4\cA_4|_{\gamma} + \h_5\cA_5|_{\gamma},$$
and for the vertical part, we obtain the ODE
\begin{equation}\label{PreL}
\frac{\dif}{\dif t}(\h_1,\h_2,\h_3,\h_4,\h_5) = (-\h_3\h_5, \h_3\h_4, \h_1\h_5 - \h_2\h_4, \h_3\h_2, -\h_3\h_1).
\end{equation}
Here $\h(t) = \sum^5_{i=1}\h_i(t)\omega^i|_{\gamma(t)}$ is the momentum expressed in the moving dual frame of reference.
This system of ODE's takes a simple form in $s$ parametrization, where we have $\h_3 = \frac{\dif s}{\dif t}>0$. It is given by
\begin{equation}\label{eq:horpartPmecs}
\frac{\dif}{\dif s}(\h_1,\h_2,\h_3,\h_4,\h_5) = \left(-\h_5, \h_4, \frac{\h_1\h_5 - \h_2\h_4}{\h_3}, \h_2, -\h_1 \right).
\end{equation}
This system can be easily integrated as follows:
\begin{equation} \label{ls}
\begin{array}{c}
\h_{1}(s)=\h_{1}(0)\cosh s-\h_{5}(0)\sinh s, \qquad
\h_{5}(s)=\h_{5}(0)\cosh s-\h_{1}(0)\sinh s, \\
\h_{2}(s)=\h_{2}(0)\cosh s+ \h_{4}(0) \sinh s, \qquad
\h_{4}(s)=  \h_{4}(0)\cosh s+ \h_{2}(0) \sinh s, \\
\h_{3}(s)=\sqrt{1-|\h_{4}(s)|^{2}-|\h_{5}(s)|^{2}},
\end{array}
\end{equation}
where the last expression follows from the Hamiltonian and the restriction $\h_3>0$. Thus we obtain two hyperbolic phase portraits (see Figure~\ref{phaseplot}).

Recall that spatial arc-length parametrization is well-defined only for the geodesics $\gamma$ whose spatial projections $\pi(\gamma)=\ul{x}(\cdot)$ do not have external cusps. Further we evaluate $\displaystyle \smax = \min\{s>0|u^3(s) = 0\}$, i.e. the minimal positive value of $s$, such that $\pi(\gamma(s)) = \bx(s)$ is a cusp point.

In the remainder of this article we use the short notation
\begin{equation}\label{eq:lamshort}
\ulh{1}:=(-\h_{1},-\h_{2}) \textrm{ and } \ulh{2}:=(\h_{5},-\h_{4}),
\end{equation}
which is not to be confused with the pair $(\blambda^{(1)},\blambda^{(2)}) \in \R^6$ given by (\ref{eq:boldlam}).

From the Hamiltonian $H = \frac{\h_3^2 + \h_4^2+\h_5^2}{2} = \frac12$, we conclude that $\|\ulh{2}\| \leq 1$. For $\|\ulh{1}\|$ we have no restrictions. This follows from $\h_1^2+\h_2^2+\h_3^2 = \gc^2$, where $\gc \in \R$.
\subsection{Computation of the First Cusp Time\label{ch:cusp}}
An arbitrary geodesic in $\Pmec$ cannot be extended infinitely in $s$-parameterization, since in the common case its spatial projection presents a cusp, where spatial arclength $s$-parametrization breaks down. In fact, for any given initial values of $\ulh{1}(0)$ and $\ulh{2}(0)$, the maximum length $\smax$ of such a geodesic,  where we have $\kappa(s)\rightarrow\infty$ as $s\uparrow s_{max}$, is given by the following theorem.
\begin{theorem}\label{sMax}
The spatial projection of a geodesic of $\Pmec$ corresponding to initial momenta $\ulh{1}(0)$, $\ulh{2}(0)$ such that $\|\ulh{2}(0)\| \leq 1$, presents a cusp first time $t^{1}_{cusp}=t(\smax)$ when $s=\smax$,
    \begin{equation}\label{smaxfrml}
        \smax = \frac{1}{2} \log{\frac{1 + \gc^2 + 2 \sqrt{\gc^2 - W^2}}{\|\ulh{2}(0) + \ulh{1}(0)\|^2}},
    \end{equation}
    with $W$ and $\gc$ given by Lemma~\!\ref{lemme:coadjorb}. For given $\ulh{1}(0)$ and $\ulh{2}(0)$, the spatial projection of the corresponding geodesic does not have a cusp for all end times iff $\|\ulh{2}(0) + \ulh{1}(0)\| = 0$.
\end{theorem}
\begin{proof}
By definition we have $\smax = \min\{s>0|u^3=0\}$. From PMP we have $u^3(s) = \h_3(s)$ for geodesics. Thus $u^3(s)=0 \desda \h_3(s)=0$. From the Hamiltonian, we have $\h_3(s) = \sqrt{1- (\h_4^2(s)+\h_5^2(s))}$, yielding $\h_3(s)=0 \desda \h_4^2(s)+\h_5^2(s)=1$. Expressions for $\h_4(s)$ and $\h_5(s)$ are given by~(\ref{ls}), so~(\ref{smaxfrml}) provides the minimal positive root of $ \h_4^2(s)+\h_5^2(s)=1$.
\end{proof}
\begin{cor}\label{cor:maxsmax}
For fixed $\|\ulh{2}(0)\|$ and $\|\ulh{1}(0)\|$, $\smax$ is maximal at those $\ulh{2}(0)$ and $\ulh{1}(0)$ such that $W=0$ and $\ulh{2}(0)\cdot\ulh{1}(0)\leq 0$.
\end{cor}
\begin{proof}
Let an angle $-\pi\leq\theta \leq \pi$ be chosen such that $\ulh{2}(0)\cdot\ulh{1}(0) = \|\ulh{2}(0)\|\|\ulh{1}(0)\|\cos{\theta}$ and $W = \|\ulh{2}(0)\|\|\ulh{1}(0)\|\sin{\theta}$.
Then along with the condition $\smax>0$, Eq.~\!\eqref{smaxfrml} yields
$\frac{\dif \smax}{\dif\theta}=0\Rightarrow\sin{\theta}=0,$
yielding three critical points, $\pm\pi$ and $0$. Comparing $\smax$ at these critical points, we get $\pm\pi$ as candidates for $\smax$ to be maximum. Checking the 2nd derivative at $\pm\pi$ we obtain $\smax$ to be maximum at $\theta=\pm\pi$. Thus $W=0$ and $\ulh{2}(0)\cdot\ulh{1}(0) = -\|\ulh{2}(0)\|\|\ulh{1}(0)\| \leq 0$.
\end{proof}
\subsection{The Exponential Map} \label{sec:ExplictFormulas} 
In order to integrate the geodesic equations, we apply Theorem~\ref{th:3}.
This provides the explicit formulas for the sub-Riemannian geodesics, which we present in the next theorem.
The sub-Riemannian geodesics are parameterized by elliptic integrals of the 1-st, 2-nd and 3-rd kind
\begin{eqnarray*}
&F(\varphi,m) = \int \limits_{0}^{\varphi} (1 - m \sin^2 \th)^{-\frac12} \, \dif \th, \qquad 
E(\varphi,m) = \int \limits_{0}^{\varphi} (1 - m \sin^2 \th)^{\frac12} \, \dif \th, \\ 
&\Pi(n,\varphi,m) = \int \limits_{0}^{\varphi}(1 - m \sin^2 \th)^{-\frac12}(1-n \sin^2 \th)^{-1} \, \dif \th.
\end{eqnarray*}
In our formulas below we use constants $W = - \h_2 \h_5 - \h_1 \h_4$, $\gc = \sqrt{\|\ulh{1}\|^2 - \|\ulh{2}\|^2+1}$.
\begin{theorem}\label{statcurv}
    Let the momentum covector be given by Eq.~\!(\ref{ls}), where $\sum_{i=1}^3 \h_i^2(0) \neq 0$. Then the spatial part of the cuspless sub-Riemannian geodesics in $\Pmec$ is given by
    \begin{equation}\label{transform}
        \bx(s) = \tR(0)^T(\tilde{\bx}(s) - \tilde{\bx}(0)),
    \end{equation}
    where $\tR(0)$ and $\tilde{\bx}(s) := (\tx(s), \ty(s), \tz(s))$ are given in terms of $\underline{\h}^{(1)}(0)$ and $\underline{\h}^{(2)}(0)$ depending on several cases.
    For all cases with $\ulh{1}(0)\neq\ulh{2}(0)$ we have
    \begin{equation}\label{tildecoordx}
        \tx(s) = \frac{1}{\gc} \int \limits_{0}^{s} \! \h_{3}(\tau) \, \dif \tau  = - \frac{i \sqrt{1-d}\sqrt{1+\gc^2}}{\gc \sqrt{2}}\left(E\left(\left(s+\frac{\varphi}{2}\right) i, M\right)-E\left(\frac{\varphi}{2} i, M\right) \right),
    \end{equation}
    where
    $M := \frac{2 d}{d-1}$, $d := \frac{\|\ulh{2}(0)+\ulh{1}(0)\|\|\ulh{2}(0)-\ulh{1}(0)\|}{1+\gc^2}\leq 1$, and $\varphi := \log \frac{\|\ulh{2}(0)+\ulh{1}(0)\|}{\|\ulh{2}(0)-\ulh{1}(0)\|}$.

    For the case $\ulh{1}(0) = \bzero$, we have
    \begin{equation}\label{arnought1}
        \tR(0) = \left(
            \begin{array}{ccc}
                0   & 0 & 1 \\
                0   & 1 & 0 \\
                -1  & 0 & 0
            \end{array}
        \right) \in \SO,
        \qquad
        \left(
            \begin{array}{c}
                \ty(s)  \\
                \tz(s)
            \end{array}
        \right)
        = \frac{-1}{\gc }
        \left(
            \begin{array}{c}
                 \h_{4}(s)  \\
                 \h_{5}(s)
            \end{array}
        \right).
    \end{equation}
    For the case $\ulh{1}(0) \neq \bzero$, we have
    \begin{equation}\label{arnought2}
        \tR(0) =
        \frac{1}{\gc}
        \left(
            \begin{array}{ccc}
                \h_1(0)                                    &      \h_2(0)                                 &    \h_{3}(0)   \\
                \gc  \frac{-\h_{2}(0)}{\|\ulh{1}(0)\|}     &  \gc  \frac{\h_{1}(0)}{\|\ulh{1}(0)\|}       &        0       \\
                \frac{-\h_{1}(0)\h_{3}(0)}{\|\ulh{1}(0)\|} &  \frac{-\h_{2}(0)\h_{3}(0)}{\|\ulh{1}(0)\|}  & \|\ulh{1}(0)\|
            \end{array}
        \right) \in \SO.
    \end{equation}
    For the case $W = 0$ along with $\ulh{1}(0) \neq \bzero$, we have
    \begin{equation} \label{help}
        \left(
            \begin{array}{c}
                \ty(s)  \\
                \tz(s)
            \end{array}
        \right) = \frac{\ulh{2}(s)\cdot\ulh{1}(0)}{\gc  \|\ulh{1}(0)\|} \left(
            \begin{array}{c}
                0   \\
                1
            \end{array}
        \right)
        .
    \end{equation}
    For $W \neq 0$ along with $\ulh{1}(0) \neq \bzero$ 
    we have
    \begin{equation}\label{intglA}
        \left(
            \begin{array}{c}
                \ty(s)  \\
                \tz(s)
            \end{array}
        \right) =
        \frac{\sqrt{\|\underline{\h}^{(2)}(s)\|^2 - W^2\gc^{-2}}}{\gc^2  \|\underline{\h}^{(1)}(0)\| \sqrt{\|\underline{\h}^{(2)}(0)\|^2 - W^2\gc^{-2}}}
        \left(
                        \begin{array}{cc}
                            \cos \tpsi(s)   & -\sin \tpsi(s)    \\
                            \sin \tpsi(s)   &  \cos \tpsi(s)
                        \end{array}
        \right)
        \left(
            \begin{array}{c}
                W \h_{3}(0)  \\
                \gc (\ulh{2}(0)\cdot \ulh{1}(0))
            \end{array}
        \right),
    \end{equation}
where 
\begin{eqnarray}
\tpsi(s) = \int \limits_{0}^{s} \! \frac{W \gc^{-1} \h_{3}(\tau)}{\|\ulh{2}(\tau)\|^2 - W^{2}\gc^{-2}} \, \dif \tau
= -\frac{W}{\gc} \frac{\sqrt{2}}{\sqrt{1+\gc^2}\sqrt{1-d}} \frac{1}{i}\big(F(i(s+\frac{\varphi}{2}),M) - F(\frac{i \varphi}{2},M)\nonumber \\
  - (1 - \frac{1}{D})(\Pi\left(\frac{M}{D},i(s +\frac{\varphi}{2}),M\right) - \Pi \left(\frac{M}{D},\frac{i \varphi}{2},M\right)) \big), \label{eq:phi}
\end{eqnarray}
with $D = 2(\frac{W^2}{\gc^2} - 1)(1+\gc^2)^{-1} (1-d)^{-1} + 1$ and $|\tpsi(s)| < \pi$, $\sign(\tpsi(s)) = \sign(W)$.
\end{theorem}
\begin{proof}
 We use Theorem~{\ref{th:3}} and apply Corollary~\ref{cor:new}. From which we have $\gamma(s) = \tilde{\gamma}(0)^{-1} \tilde{\gamma}(s)$, where $m(\tilde{\gamma}(s))$ relates to $\bh(s)$ via Eq.~(\ref{eq:representative}). This provides~(\ref{transform}).
    For the most general case, assuming non-vanishing denominators throughout, we see that when choosing \eqref{arnought2} and
    $\tilde{\bx}(0) := \frac{1}{\gc^2 \|\ulh{1}(0)\|} \left(0, W \h_{3}(0),\gc (\ulh{1}(0)\cdot\ulh{2}(0))\right)^T$, equation \eqref{eq:representative} is satisfied in the initial moment $s=0$. Then solving \eqref{eq:representative} for $\tx$, $\ty$ and $\tz$ we obtain $ \tx(s) = \frac{1}{\gc} \int \limits_{0}^{s} \! \h_{3}(\tau) \, \dif \tau$ for $\tx(s)$ and for $(\ty(s), \tz(s))$ we obtain the following system:
  \begin{equation}\label{A}
  \begin{array}{l}
    \left(
            \begin{array}{c}
                \ty'(s)  \\
                \tz'(s)
            \end{array}
        \right) = A(s)
        \left(
            \begin{array}{c}
                \ty(s)  \\
                \tz(s)
            \end{array}
        \right), \\
    \textrm{ with } A(s) = \frac{1}{\|\ulh{2}(0)\|^2 - \frac{W^{2}}{\gc^2 }}
         \left(
            \begin{array}{cc}
               \ulh{2}(s) \cdot \ulh{1}(s)    & - \frac{W}{\gc} \h_{3}(s)  \\
               \frac{W}{\gc} \h_{3}(s) & \ulh{2}(s)\cdot \ulh{1}(s)
            \end{array}
        \right).
        \end{array}
  \end{equation}
    Note that $A(s)$ and $A(s')$ commute, and Wilcox formula \cite{Wilcox} yields the result.

    Clearly, the formulas are not valid when denominators vanish. Hence we do the whole procedure keeping in mind the special cases (\ref{arnought1}), (\ref{help}) right from the start and get the required results.  

    Regarding (\ref{intglA}) we note that matrix $e^{\int_{0}^{s}A(\tau){\rm d}\tau}$ can be computed explicitly. One has
$$
                e^{\int \limits_{0}^{s} \! A(\tau) \, \dif \tau} =
                                \sqrt{\frac{\|\underline{\h}^{(2)}(s)\|^2 - W^2\gc^{-2}}{\|\underline{\h}^{(2)}(0)\|^2 - W^2\gc^{-2}}}
                    \left(
                        \begin{array}{cc}
                            \cos \tpsi(s)   & -\sin \tpsi(s)    \\
                            \sin \tpsi(s)   &  \cos \tpsi(s)
                        \end{array}
                    \right),
$$
with $\tpsi(s) = \int \limits_{0}^{s} \! \frac{W \gc^{-1} \h_{3}(\tau)}{\|\ulh{2}(s')\|^2 - W^{2}\gc^{-2}} \, \dif \tau$.
From the first integrals one can deduce that $\forall s\in (0,\smax): \|\ulh{2}(s)\|^2 - W^{2}\gc^{-2} >0$. Note $\|\ulh{2}(s)\|^2 = W^{2}\gc^{-2} \desda s= \smax \wedge \ulh{1}(\smax)\cdot\ulh{2}(\smax)=0$. By direct computation~(\ref{eq:phi}) follows. Moreover, by Lemma~4.13 in~\cite{GhoshDuitsArxiv} we have $|\tpsi(s)|<\pi$ for all $s \leq \smax$ and since $\frac{\gc^{-1} \h_3(s)}{\|\ulh{2}(s)\|^2 - W^{2}\gc^{-2}} \geq 0$ we have $\sign(\tpsi(s)) = \sign(W)$.

The remaining part is to prove that $ \tx(s) = \frac{1}{\gc} \int \limits_{0}^{s} \! \h_{3}(\tau) \, \dif \tau$ can be integrated in terms of elliptic integrals, as presented in~(\ref{tildecoordx}). This is done by the following computation:
\begin{eqnarray*}
\tx(s) = \int_0^s \frac{\h_3(\tau)}{\gc} \dif \tau =  \frac{1}{\gc}\int_0^s\sqrt{1 - \|\ulh{2}(\tau)\|^2}\dif\tau =
\frac{\sqrt{1+\gc^2}}{\gc \sqrt{2}} \int_0^s \sqrt{1 - c_1 \cosh(2 \tau) - c_2 \sinh(2\tau)} \dif \tau,
\end{eqnarray*}
where $c_1 := \frac{1}{1+\gc^2}(\|\ulh{2}(0)\|^2+\|\ulh{1}(0)\|^2)$ and $c_2 := \frac{2}{1+\gc^2}\ulh{2}(0)\cdot\ulh{1}(0)$. \\
Denoting $d := \sqrt{c_1^2 - c_2^2}$, $\varphi := \frac12 \log\frac{c_1+c_2}{c_1-c_2}$, $M := \frac{2 d}{d-1}$ and $\th := i \tau + \frac{i \varphi}{2}$ we can express
\[
\begin{array}{ll}
\tx(s) & = \frac{\sqrt{1+\gc^2}}{\gc \sqrt{2}} \int_0^s \sqrt{1 - d \cos(2 i \tau + i \varphi)} \dif \tau
= \frac{\sqrt{1-d}\sqrt{1+\gc^2}}{i \gc \sqrt{2}} \int\limits_{\frac{i \varphi}{2}}^{i(s+ \frac{\varphi}{2})}\sqrt{1 - M \sin^2(\th)}\; \dif \th \\
       & = - \frac{i \sqrt{1-d}\sqrt{1+\gc^2}}{\gc \sqrt{2}}(E((s+\frac{\varphi}{2}) i, M) - E((\frac{\varphi}{2}) i, M)),
\end{array}
\]
where we note that
$d \leq \frac{1}{1+\gc^2}\frac12 (\|\ulh{2}(0)-\ulh{1}(0)\|^2+\|\ulh{2}(0) + \ulh{1}(0)\|^2) \leq 1$.
\end{proof}
\begin{figure}
\centerline{
\includegraphics[width=0.6\linewidth]{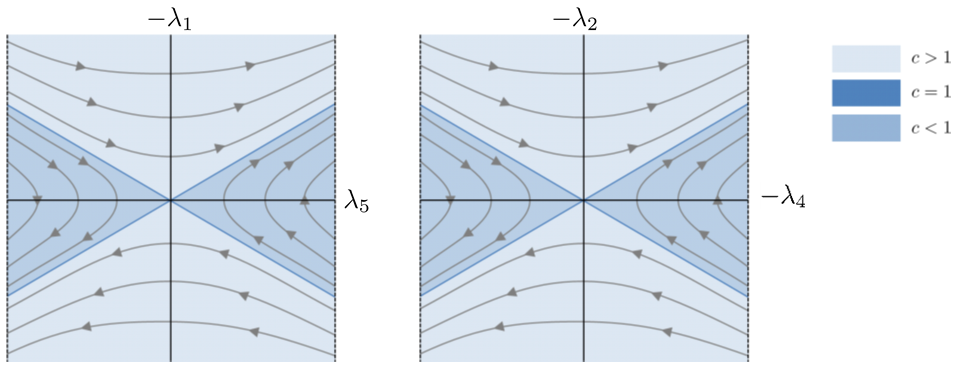}
}
\caption{Phase portraits corresponding to the components of $\ulh{1}$, $\ulh{2}$ satisfying the second order differential equation $\frac{\dif^2}{\dif s^2}\ulh{1}(s) = \ulh{1}(s)$ along the geodesics. Several orbits are shown with arrows.}
\label{phaseplot}
\end{figure}
\begin{corollary} \label{cor:aftermain}
The exponential map $Exp:\mathcal{D}_0 \to \R^{3} \rtimes S^2$ defined in Definition~\ref{def:exp}
is given by
\[
Exp(\lambda(0),L)=(\ul{x}^*(L), \ul{n}^{*}(L)),
\]
where $\ul{n}^{*}(L)=\frac{d}{ds}\ul{x}^{*}(s)|_{s=L}$ and where the spatial part of the geodesic
$\ul{x}^{*}(L)=(x^*(L),y^*(L),z^*(L))$ is explicitly given by (\ref{transform}). Here the tangent equals
\[
\frac{d}{ds}\ul{x}^{*}(s)|_{s=L}=\ul{x}^{*\, '}(L)= (\tilde{R}(0))^T \tilde{\ul{x}}'(L),
\]
with $\tilde{R}(0)$ given by (\ref{arnought1}) and (\ref{arnought2}),
and $\tilde{x}'(L)= \frac{1}{\gothic{c}} \lambda_{3}(L)$ with $\lambda_{3}(L)$ given by (\ref{ls}),
and $(\tilde{y}'(L),\tilde{z}'(L))^T= A(L) \; (\tilde{y}(L),\tilde{z}(L))^T$ with $A(L)$ given by (\ref{A}) and
$\tilde{y}(L),\tilde{z}(L)$ given by (\ref{intglA}).
\end{corollary}

\subsection{Geometric Properties of the Stationary Curves}\label{ch:geom}
Let a stationary curve of $\Pcurve$ in $\R^3$ be given by $\bx : [0, L] \rightarrow \R^3$ parameterized by arclength denoted by $s$. Let the unit tangent, the unit normal and the unit binormal for this curve be given by $\bT$, $\bN$, and $\bB$ respectively.
Let $\kappa$ and $\tau$ denote curvature and torsion, then the
 Frenet-Serret equations are
\begin{equation}\label{FSframe}
    \frac{\dif}{\dif s} \left(
        \begin{array}{c}
            \bT(s)  \\
            \bN(s)  \\
            \bB(s)
        \end{array}
    \right) = \left(
        \begin{array}{ccc}
            0           & \kappa(s) & 0         \\
            - \kappa(s) & 0         & \tau(s)   \\
            0           & - \tau(s) & 0
        \end{array}
    \right) \left(
        \begin{array}{c}
            \bT(s)  \\
            \bN(s)  \\
            \bB(s)
        \end{array}
    \right).
\end{equation}

Let us first study the curvature and signed torsion of the spatial projection of the sub-Riemannian geodesics.
Let us recall the first integral constants $W = - \h_2 \h_5 - \h_1 \h_4$ and $\gc = \sqrt{\|\ulh{1}\|^2 - \|\ulh{2}\|^2+1}$ in Lemma~\ref{lemme:coadjorb}.
Furthermore, we have that
\begin{equation}\label{one}
R' = \frac{dt}{ds}\dot{R}= \h_{3}^{-1}\dot{R} = \h_{3}^{-1} R\left(
            \begin{array}{ccc}
                     0     &     0      &    \h_5 \\
                     0     &     0      & -\h_4    \\
                -\h_5  & \h_4  &    0
            \end{array}
        \right).
\end{equation}
and therefore {\small
\begin{equation} \label{two}
\bx '(s)=R(s)\be_z \Rightarrow \bx''(s)=R(s)\left(
                                                 \begin{array}{c}
                                                  \frac{\h_5(s)}{\h_{3}(s)} \\
                                                   -\frac{\h_4(s)}{\h_{3}(s)} \\
                                                        0      \\
                                                 \end{array}
                                               \right)
                    \Rightarrow \bx '''(s)=R(s)\left(
                                                  \begin{array}{c}
                                                    \frac{d}{ds} \left(\frac{\h_5(s)}{\h_3(s)}\right) \\
                                                    -\frac{d}{ds} \left(\frac{\h_4(s)}{\h_3(s)}\right) \\
                                                      -\frac{\h_4^2(s) +\h_5^2(s)}{\h_3^2(s)}   \\
                                                  \end{array}
                                                \right).
\end{equation}
}
\begin{theorem}\label{ThVM}
The absolute curvature and the signed torsion of a stationary curve of $\Pcurve$ are given by
\begin{equation}\label{eq:curvandtor}
 \kappa= \frac{\sqrt{\h_4^2+ \h_5^2}}{\h_3}=\frac{\sqrt{1-\h_3^2}}{\h_3}, \ \qquad \tau = \frac{W}{\h_4^2+\h_5^2},
\end{equation}
with momentum components $\h_i$ given by (\ref{ls}).
We have the following fundamental relation between curvature and torsion
\begin{equation} \label{fund}
\tau(s) \, \kappa^2(s) = W \, (t'(s))^2.
\end{equation}
The torsion is bounded as follows
\begin{equation}\label{torbound}
        |W| \leq |\tau(s)| \leq \frac{2|W|}{\sqrt{(1 - \gc^2)^2 + 4 W^2} + 1 - \gc^2} \mbox{ for all } 0 \leq s \leq L \leq \smax.
    \end{equation}
\end{theorem}
\begin{proof}
In the proof we use the following properties of norm, inner product and cross product:
\begin{eqnarray*}
\forall_{R \in \SO} \forall_{ \ul{a}, \ul{b} \in \R^3}\,:\,  \|R \ul{a}\| = \|\ul{a}\|, \quad
                                                      (R \ul{a})\cdot(R \ul{b}) = \ul{a}\cdot \ul{b}, \quad
                                                      (R \ul{a})\times(R \ul{b}) = R (\ul{a}\times \ul{b}).
\end{eqnarray*}
First part follows by straightforward computation via~(\ref{one}) and~(\ref{two}).
By definition we have $\kappa(s) = \|\bx''(s)\|$. Thus by~(\ref{two}) and the Hamiltonian $H = \frac12(\h_3^2+\h_4^2+\h_5^2) = \frac12$ we obtain, that the curvature satisfies~(\ref{eq:curvandtor}). For arclength parametrized curve $\bx(s)$ in $\R^3$ the torsion is given by $\tau = \frac{(\bx' \times \bx'')\cdot \bx'''}{\|\bx' \times \bx''\|^2}$ (see e.g. \cite{Spiv75b}). Thus by~(\ref{two}) we have $(\bx' \times \bx'')\cdot \bx''' = \frac{W}{\h_3^2}$ and $\|\bx' \times \bx''\|^2 = \frac{\h_4^2+\h_5^2}{\h_3^2}$, and thereby  $\tau$ satisfies~(\ref{eq:curvandtor}). Eq.~\!(\ref{fund}) follows by $\lambda_3=\frac{ds}{dt}$ and~\!(\ref{eq:curvandtor}).

In order to prove bounds~(\ref{torbound}) we use expression for torsion $\tau = \frac{W}{\h_4^2+\h_5^2} \Rightarrow |\tau| = \frac{|W|}{\h_4^2+\h_5^2}$. The lower bound in~(\ref{torbound}) holds since $\h_4^2+\h_5^2\leq 1$ due to the Hamiltonian.

To prove the upper bound, we show that $\h_4^2(s)+\h_5^2(s) \geq \frac12 \left( \sqrt{(1-\gc^2)^2+ 4 W^2} + 1 -\gc^2\right)$ for all $s \in [0,\smax]$. To obtain the last inequality, we solve $\frac{\dif \left(\h_4^2(s)+\h_5^2(s)\right)}{\dif s} = 0$ using~(\ref{ls}):
\[
\smin := \arg\min_{s \in [0,\smax]}\{\h_4^2(s)+\h_5^2(s)\}=
\begin{cases}
0, \textrm{ if } -\h_5(0) \h_1(0) + \h_2(0) \h_4(0) >0, \\
\frac12 \log \frac{\|\ulh{2}(0) - \ulh{1}(0)\|}{\|\ulh{2}(0) + \ulh{1}(0)\|}, \textrm{ if } -\h_5(0) \h_1(0) + \h_2(0) \h_4(0) \leq 0.
\end{cases}
\]
Thus, evaluation $|\tau(\smin)| = \frac{|W|}{\h_4^2(\smin)+\h_5^2(\smin)}$ gives the upper bound in~(\ref{torbound}). Here we use identity
$\|\ulh{2}(0)+\ulh{1}(0)\|\|\ulh{2}(0)-\ulh{1}(0)\| = 
\sqrt{(1-\gc^2)^2 + 4 W^2}$.
\end{proof}
\begin{cor}\label{cor:torsion}
The cuspless spatial projections of sub-Riemannian geodesics of $\Pmec$ with $\sum_{i=1}^3\h_i^2(0)\neq 0$ (i.e. the stationary curves of $\Pcurve$) are planar if and only if $W=0$.
\end{cor}
Next we show that when taking the end conditions to be co-planar, one gets $W=0$ implying that planar curves are \emph{the only} cuspless geodesic of problem $\Pmec$ connecting these conditions.
\begin{theorem}\label{coplan}
Let $\bx:[0,s_{max}]\rightarrow\R^3$ be the spatial part of a cuspless sub-Riemannian geodesic of $\Pmec$ given by Theorem \ref{statcurv}, i.e. let $\bx$ be a stationary curve of $\Pcurve$. Then for any $s\in (0,s_{max}]$, one has that $\be_z$, $\bx(s)$ and $\bx'(s)$ are coplanar if and only if $\bx$ is a planar curve, i.e. $$\be_z\cdot(\bx(s)\times \bx'(s))=0\Leftrightarrow W=0.$$
\end{theorem}
\begin{proof}
If $W=0$ it follows by Corollary~\ref{cor:torsion} that the curve is planar. By Theorem~\ref{statcurv} we indeed get that $W=0\Rightarrow\ty(s)\equiv 0$ and  $\ty(s)\equiv0\Rightarrow \tilde{\bx}(0)\cdot(\tilde{\bx}(s)\times \tilde{\bx}'(s))=0\Rightarrow\be_z\cdot(\bx(s)\times \bx'(s))=0$.

Now we focus on the other direction of the implication. Let us consider curve $\bx':[0,s_{max}]\rightarrow S^2$. It can have a minimum curvature of $1$ if it aligns with a great circle on $S^2$. By Eq.\!~(\ref{two}) and (\ref{eq:curvandtor}) it follows that
the geodesic curvature $\mathcal{K}_{tan}$ of $\bx'(\cdot)$ is given as
\[
\mathcal{K}_{tan}(s) := \frac{\|\bx''(s)\times\bx'''(s)\|}{\|\bx''(s)\|^3} = \frac{\sqrt{\kappa^6(s)+\left(\kappa^2(s)+1\right)^2 W^2}}{\kappa^3(s)} = \sqrt{1+\frac{\left(\kappa^2(s)+1\right)^2W^2}{\kappa^6(s)}}.
\]
Thus $W\neq 0 \Rightarrow \mathcal{K}_{tan}(s)>1$ for any $s\in(0,s_{max})$. The curve $\bx'$ gets aligned with great circles only at cusp points where $\kappa(s)=\infty$ which never occurs in an interior point. Thus the curve $\ul{n}=\bx'$ can intersect a great circle on $S^2$ at most at two points. Therefore, any three points along this curve can never lie simultaneously on a plane passing through the origin. Thus $$\forall_{\tau,s>0}:\bx'(0)\cdot(\bx'(\tau)\times \bx'(s))\neq0 \Rightarrow \int\limits_0^s\left(\bx'(0)\cdot(\bx'(\tau)\times\ \bx'(s))\right)\dif\tau\neq0 \Rightarrow \bx'(0)\cdot(\bx(s)\times \bx'(s))\neq0. $$
So $W\neq0\Rightarrow\be_z\cdot(\bx(s)\times \bx'(s))\neq 0$. In turn this leads to $\be_z\cdot(\bx(s)\times \bx'(s))=0\Rightarrow W=0.$
\end{proof}

The following corollaries relate the planar cuspless sub-Riemannian geodesics in $(\SE,\Delta,\mathcal{G}_{1})$ to those in $(\SEtwo,\tilde{\Delta},\tilde{\mathcal{G}}_{1})$ with $\tilde{\Delta}={\rm ker}\{-\sin{\theta}\dif x + \cos{\theta}\dif y\}$ and $\tilde{\mathcal{G}}_{1}=(\cos{\theta}\dif x + \sin{\theta}\dif y)\otimes (\cos{\theta}\dif x + \sin{\theta}\dif y) + \dif\theta\otimes\dif\theta)$, cf.\!~\cite{DuitsSE2}.
\begin{cor}\label{cor:cop}
Given admissible coplanar end conditions for $\Pcurve$, the unique cuspless stationary curve connecting them is planar.
\end{cor}
Now by the global optimality results in \cite{Boscain2013a,DuitsSE2} we have the following corollary.
\begin{cor}\label{cor:embedSe2}
Let $\tilde{\mathcal{R}}$ denote the range of the exponential map of cuspless geodesics in $\SEtwo$, which coincides with the set of admissible end conditions of $\Pcurve$ in $\SEtwo$. Then the set $\{\left(\bx_1, (\frac{x_1}{\sqrt{x_1^2+y_1^2}}\sin\theta_1, \frac{y_1}{\sqrt{x_1^2+y_1^2}}\sin\theta_1, \cos\theta_1)^T\right)|(\bx_1, \theta_1)\in \tilde{\mathcal{R}}\} \subset \mathcal{R}$ (recall (\ref{eq:rangePcurveSE3})) is a set of end conditions admitting a unique global cuspless minimizer of $\Pcurve$.
\end{cor}
Next we show that the sub-Riemannian geodesics do not self intersect or roll up, despite the fact that the absolute curvature $\kappa(s)\rightarrow\infty$ as $s\uparrow \smax$.
\begin{cor}\label{cor:mon}
The cuspless spatial projections of sub-Riemannian geodesics of $\Pmec$ with $\h_1^2(0)+\h_2^2(0)+\h_3^2(0) \neq 0$, have a monotonically increasing component along  $\gc^{-1}\left(\h_1(0) \be_x + \h_2(0)\be_y + \h_3(0)\be_z\right)$. Hence they do not self intersect or roll up.
\end{cor}
\begin{proof}
From Eq.~\!(\ref{transform}) we have $\ul{e}_{\tilde{x}}= (\tilde{R}(0))^T \ul{e}_x=\gc^{-1}\left(\h_1(0) \be_x + \h_2(0)\be_y +  \h_3(0)\be_z\right)$. By Theorem~\ref{statcurv} we have $\tx'(s)>0$ for all $s\in(0,s_{max})$ and the result follows.
\end{proof}
Next we want to study bounds on the set $\mathcal{R}$, recall (\ref{eq:rangePcurveSE3}).
 Our numerical investigations clearly show that the spatial part of all points in $\mathcal{R}$ is contained in the half space $z\geq 0$, and that the plane $z=0$ is only reached with U-shaped planar geodesics (i.e. $W=0$, $\gothic{c}<1$) at $s=\smax$.
These numerical observations inspired
 us to find the natural generalization
of formal results in \cite[Thm.7\&8]{DuitsSE2} on the $\SEtwo$-case.
We partly succeeded as we show in the next three corollaries, which provide bounds on the set $\mathcal{R}$.
\begin{cor} \label{th:A}
Let $s \mapsto \gamma(s)=(x(s),y(s),z(s),R(s))$ be a sub-Riemannian geodesic in $(\SE,\Delta,\mathcal{G}_{1})$ with $\h_6=0$ and $\h_1^2(0)+\h_2^2(0)+\h_3^2(0) \neq 0$, departing from $e=(\ul{0},I)$, s.t. the spatial
projection is cuspless.

If $\ulh{1}(0) \cdot \ulh{2}(0) \geq 0$, then $z(s) >0$ for all $s \in (0,s_{max})$.

If $\ulh{1}(0) \cdot \ulh{2}(0) < 0$, then $z(s) >0$ for all $s \in (0,s_{m})$, with
$$s_m=\begin{cases}
\log\frac{|\h_4(0) + \h_5(0)|}{|\h_4(0) - \h_5(0)|}, \textrm{ if } \ulh{1}(0) = -\ulh{2}(0),\\
\log\frac{\|\ulh{2}(0)+\ulh{1}(0)\|}{\|\ulh{2}(0)-\ulh{1}(0)\|}, \textrm{ otherwise. }\end{cases}$$
\end{cor}
\begin{proof}
If $\ulh{1}(0) \cdot \ulh{2}(0) \geq 0$, then $\ulh{1}(\tau) \cdot \ulh{2}(\tau) \geq 0$ for all $\tau \in (0,s_{max})$, which implies $A(\tau) \geq 0$ for all $\tau \in (0,s_{max})$, where we recall Eq.~\!(\ref{A}). Now by Theorem~\ref{statcurv}, we have
\[
\begin{array}{ll}
z(s) &=(\ul{x}(s),\ul{e}_{z})=((\tilde{R}(0))^T(\tilde{\ul{x}}(s)-\tilde{\ul{x}}(0)),\ul{e}_{z})= (\tilde{\ul{x}}(s)-\tilde{\ul{x}}(0),\tilde{R}(0)\ul{e}_z)=(\tilde{\ul{x}}(s)-\tilde{\ul{x}}(0),\tilde{\ul{x}}'(0)) \\
 &= \tilde{x}(s) \tilde{x}'(0) \; +\; \left(\; (e^{\int \limits_{0}^s A(\tau)\, {\rm d}\tau}-I) \left(
 \begin{array}{c}
 \tilde{y}(0) \\
 \tilde{z}(0)
 \end{array}
 \right), A(0)\, \left(
 \begin{array}{c}
 \tilde{y}(0) \\
 \tilde{z}(0)
 \end{array}
 \right)\; \right) \\
  &= \gothic{c}^{-2} \lambda_{3}(s)\, \int \limits_{0}^{s} \lambda_3(\tau)\, {\rm d}\tau \; +\; \left(\; (A(0))^T(e^{\int \limits_{0}^s A(\tau)\, {\rm d}\tau}-I) \left(
 \begin{array}{c}
 \tilde{y}(0) \\
 \tilde{z}(0)
 \end{array}
 \right),\, \left(
 \begin{array}{c}
 \tilde{y}(0) \\
 \tilde{z}(0)
 \end{array}
 \right)\; \right).
\end{array}
\]
As cusps do not occur on the interior of spatially projected curve $\ul{x}(\cdot)=(x(\cdot),y(\cdot),z(\cdot))$, the first term $\gothic{c}^{-2} \h_{3}(s)\, \int \limits_{0}^{s} \h_3(\tau)\, {\rm d}\tau$ is strictly positive
for all $s \in (0,s_{max})$. Regarding the second term, we note that $A(\tau_{1})$ and $A(\tau_2)$ commute for all $\tau_{1},\tau_2>0$ and each $A(\tau)$ can be diagonalized.
Now $A^{T}(0)= C \, A^{-1}(0)$ for some $C>0$. Thereby, both operators $A^T(0)$ and $\left(e^{\int_{0}^s A(\tau)\, {\rm d}\tau}-I\right)$ commute, have a common eigensystem, and are either semi-positive definite or semi-negative definite, thus
\[
\begin{array}{c}
z(s) > \left(\; (A(0))^T(e^{\int \limits_{0}^s A(\tau)\, {\rm d}\tau}-I) \left(
 \begin{array}{c}
 \tilde{y}(0) \\
 \tilde{z}(0)
 \end{array}
 \right),\, \left(
 \begin{array}{c}
 \tilde{y}(0) \\
 \tilde{z}(0)
 \end{array}
 \right)\; \right) \geq 0 \desda \\
 \desda
 \left\{
 \begin{array}{l}
 A^T(0)\geq 0 \textrm{ and }e^{\int \limits_{0}^s A(\tau)\, {\rm d}\tau}-I\geq 0, \\
 \textrm{or }
 A^T(0)\leq 0 \textrm{ and }e^{\int \limits_{0}^s A(\tau)\, {\rm d}\tau}-I\leq 0
 \end{array}
 \right.
 \desda
  \left\{
 \begin{array}{l}
 \ulh{1}(0) \cdot \ulh{2}(0)\geq 0, \\
 \textrm{or }
 \ulh{1}(0) \cdot \ulh{2}(0)\leq 0 \textrm{ and }m(s) \cos (\tpsi(s)) \leq 1,
 \end{array}
 \right.
\end{array}
\]
where the scalar multiplier
\[
m(s)= \sqrt{\frac{1-|\lambda_{3}(s)|^2 -W^{2}\gothic{c}^{-2}}{1-|\lambda_{3}(0)|^2 -W^{2}\gothic{c}^{-2}}} \leq 1 \desda \lambda_{3}(s) \geq \lambda_{3}(0)\ ,
\]
comes from (\ref{A}). Now $s_m$ is chosen as the 1st positive root of $\lambda_3(s)=\lambda_{3}(0)$ and the result follows as $m(s)\leq 1 \Rightarrow m(s) \cos (\tpsi(s))-1 \leq 0$.
\end{proof}
\begin{cor}\label{cor:zpw0}
Let $W=0$ then all cuspless sub-Riemannian geodesics in $(\SE,\Delta,\mathcal{G}_{1})$ with $\h_6=0$  and $\h_1^2(0)+\h_2^2(0)+\h_3^2(0) \neq 0$, departing from $e=(\ul{0},I)$ stay in the upper half space $z\geq 0$.
\end{cor}
\begin{proof}
If $W=0$ the spatial part of such a sub-Riemannian geodesic is coplanar by Theorem~\ref{coplan}. Application of Corollary~\ref{cor:embedSe2} and \cite[Thm.7\&8]{DuitsSE2} yields the result.
 \end{proof}

\begin{cor} \label{th:B}
Let $s \mapsto \gamma(s)=(x(s),y(s),z(s),R(s))$ be a sub-Riemannian geodesic in $(\SE,\Delta,\mathcal{G}_{1})$ with $\h_6=0$ and $\h_1^2(0)+\h_2^2(0)+\h_3^2(0) \neq 0$, departing from $e=(\ul{0},I)$, s.t. its spatial projection does not have (interior) cusps. Assume it departs from a cusp and ends towards a cusp, i.e. $\h_{3}(0)=0=\h_{3}(\smax)$, where $\smax>0$ by definition.

Then
$z(s)>0$ for all $s \in (0,\smax)$,
and
$z(s)=0$ $\desda (W=0 \textrm{ and } s \in \{0,\smax\} \textrm{ and }\gothic{c}<1$).
\end{cor}
\begin{proof}
If $W \neq 0 $ and $\lambda_{3}(0)=0$, then $\|\ulh{2}(0)\|=1$ and $\|\ulh{1}(0)\|=\gothic{c}$, and
by Theorem~\ref{statcurv} one has
\[
z(s)=\tilde{z}(s)-\tilde{z}(0)= (m(s) \cos \tpsi(s)-1)\frac{\ulh{1}(0)\cdot \ulh{2}(0)}{\gothic{c}^2} \quad \textrm{   for all } s \in [0,\smax],
\]
with $m(s)<1$ if $s<\smax$ and $\ulh{1}(0)\cdot \ulh{2}(0)<0$ mandatory for $\smax>0$ in case $\h_{3}(0)=0$.
Now $m(\smax)=1$ but even then due to  $\tpsi(s)\neq0$ we have $W\neq 0  \Rightarrow z(\smax)\neq 0$, recall (\ref{eq:phi}).

If $W=0$ and $\h_{3}(0)=0$ then by Eq.~\!(\ref{help}) in Theorem~\ref{statcurv}, we have for this specific case,
\[
z(s)=\tilde{z}(s)-\tilde{z}(0)= \frac{(\ulh{2}(s)- \ulh{2}(0))\cdot \ulh{1}(0)}{\gothic{c}^2}= \frac{1}{\gc^2} \left(\gothic{c}^2 \sinh s -\gothic{c} (\cosh s-1)\right).
\]
Now only for $\gothic{c}<1$ we find two nonnegative roots $s=0$ and $s=\log \frac{1+\gothic{c}}{1-\gothic{c}}=s_{max}$. The parabola corresponding to the quadratic equation
arising when setting $p=e^s$ and multiplying with $p$ is a parabola that opens downward so that $z(s)>0$ if $s \in (0,s_{max})$.
\end{proof}
\begin{cor}\label{cor:10}
Let $(\bx_1,-\be_z)$ be the end condition of $\Pcurve$ with the initial condition $(\bzero,\be_z)$. Then, a solution to problem $\mathbf{P_{curve}}$ exists if and only if $\bx_1\cdot\be_z=0$. Moreover, this condition is only possible for curves departing from a cusp and ending in a cusp.
\end{cor}
\begin{proof}\
Let $\bx$ be a solution to problem $\Pcurve$ with $\bx'(0)=-\bx'(L)$ for some $0<L\leq \smax$. So we have $\tilde{\bx}'(0)=-\tilde{\bx}'(L)$, which implies $\tx'(0)=-\tx'(L)$. But this is possible if and only if $\tx'(0)=0=\tx'(L)$, which is possible only if $\|\ulh{2}(0)\|=1$ and $L=\smax$, i.e., if the geodesic both starts and ends at a cusp. Then $z(\smax)=0$ (see Corollary~\!\ref{th:B}).
\end{proof}

\subsection{Symmetries of the Exponential Map}\label{ch:symm}
We now describe the symmetries of the exponential map of $\Pcurve$, recall Definition~\ref{def:exp}.
Here we are interested in the symmetries that retain curvature and torsion along the curve and preserve direction of time (i.e. we do not consider the symmetries involving time inversion $s\mapsto L-s$, cf.~\cite{Moiseev}). 
From the conservation law
\begin{equation}
(\h_1(s))^2+(\h_2(s))^2 - ((\h_4(s))^2+(\h_5(s))^2) = \gc^2-1
\end{equation}
and Eq.~(\ref{eq:curvandtor}) in Theorem~\ref{ThVM}, we deduce the following corollary.
\begin{cor}\label{AllSym}
Let $P\in\R^{6\times 6}$ be given by
\begin{equation}
P =
\left(
  \begin{array}{cc|cc}
    Q & 0 & 0 & 0 \\
    0 & 1 & 0 & 0 \\ \hline
    0 & 0 & \det(Q)Q & 0 \\
    0 & 0 & 0 & 1 \\
  \end{array}
\right)
\end{equation}
with $Q\in O(2)$ arbitrary. Then we have the following symmetry property of the exponential map:
\[
\widetilde{Exp}_e(\blambda(0) P^T,l) = \left(\bzero,\left(
                          \begin{array}{cc}
                            Q & 0 \\
                            0 & 1
                          \end{array}
                        \right)
\right)\cdot \widetilde{Exp}_e(\blambda(0),l)\cdot\left(\bzero,\left(
                          \begin{array}{cc}
                            Q^T & 0 \\
                             0  & 1
                          \end{array}
                        \right)
\right).
\]
Here $\blambda = (\h_1,\ldots,\h_6)$ and the group product $\cdot$ is on the Euclidean group $E(3)$.
\end{cor}
\begin{proof} Proof can be found in~\cite{GhoshDuitsArxiv}.
\end{proof}
\begin{figure}[ht]
\begin{minipage}[b]{0.41\linewidth}
\centering
\includegraphics[width=\textwidth]{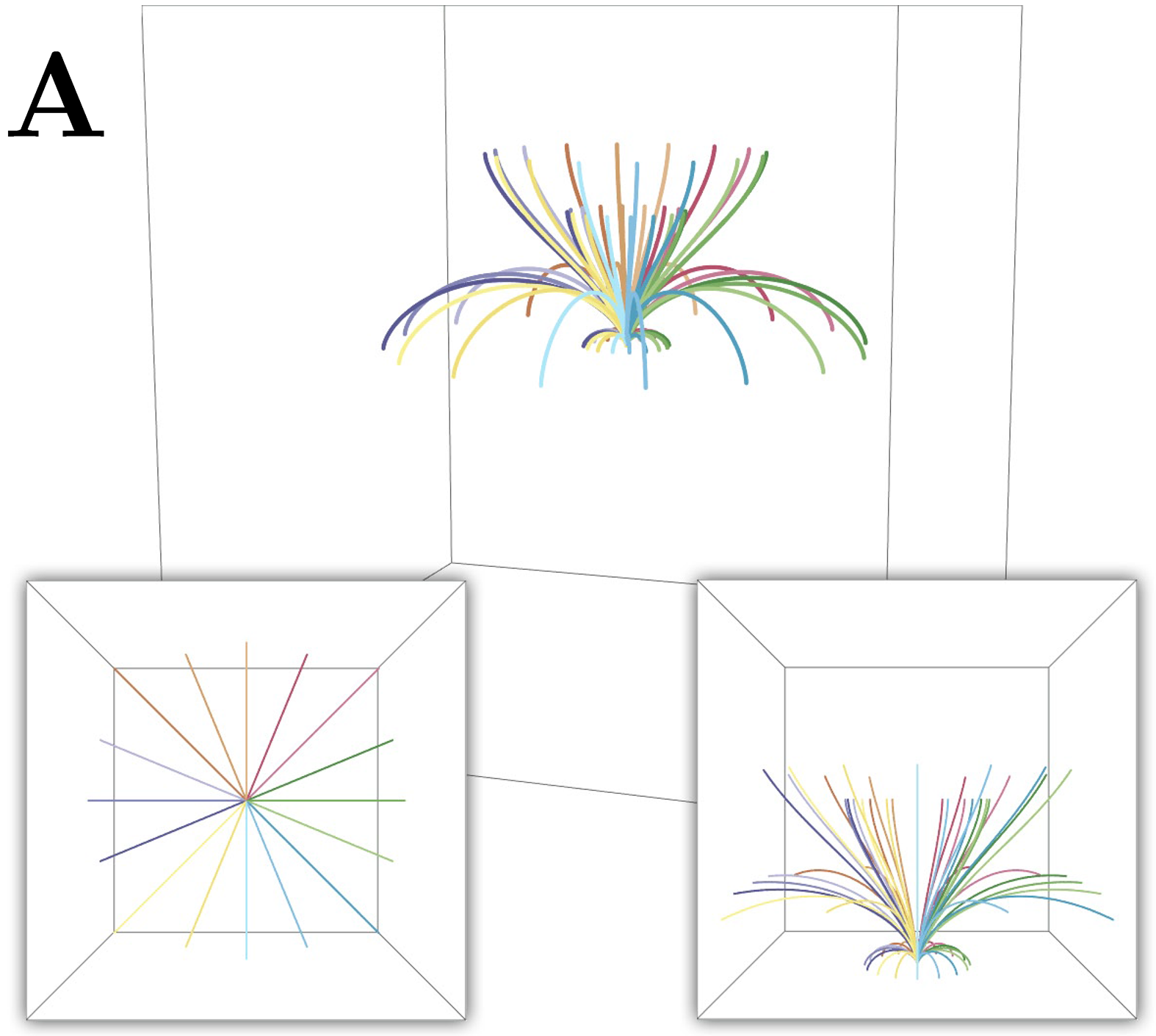}
\end{minipage}
\hspace{0.01\linewidth}
\begin{minipage}[b]{0.32\linewidth}
\centering
\includegraphics[width=\textwidth]{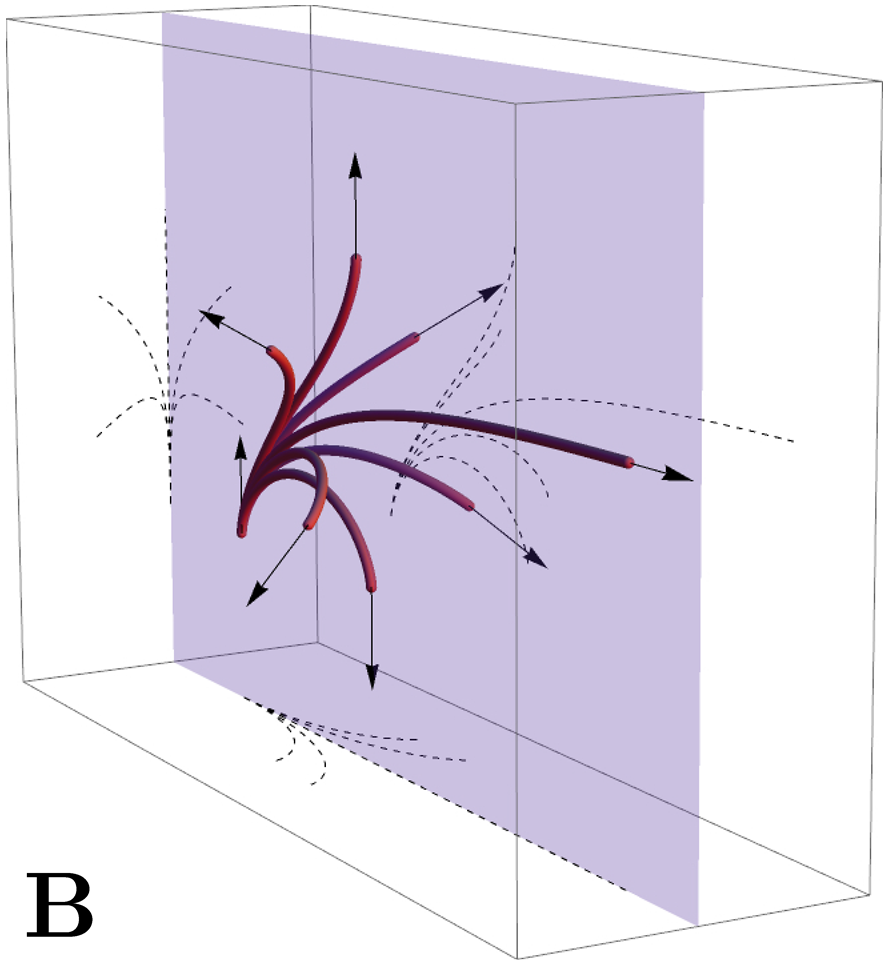}
\end{minipage}
\centering
\caption{\textbf{A:} Rotational symmetries in case of several planar cuspless sub-Riemannian geodesics of problem $\mathbf{P_{mec}}$ departing from $e$ in direction of $\be_z$.
\textbf{B:} Reflectional symmetry of certain cuspless geodesics of $\mathbf{P_{mec}}$. These curves are produced by rotating $\ulh{1}(0)$ by certain angles while keeping $\ulh{2}(0)$ fixed. The plane of reflection contains the middle curve with $\ulh{1}(0)$ parallel to $\ulh{2}(0)$.}
\label{symmetries}
\end{figure}

Figure \ref{symmetries} depicts both the rotational and reflectional symmetries of the cuspless sub-Riemannian geodesics of problem $\Pmec$. To generate the figures, we have set
\begin{equation}
\ulh{2}(0) = \|\ulh{2}(0)\|(\cos\theta,\sin\theta)^T \mbox{ and } \ulh{1}(0) = \|\ulh{1}(0)\|(\cos(\theta-\Theta),\sin(\theta-\Theta))^T.
\end{equation}
Here $\Theta$ denotes the angle between $\ulh{2}(0)$ and $\ulh{1}(0)$. For both these figures we fixed $\|\ulh{2}(0)\|$ and $\|\ulh{1}(0)\|$. For Figure \ref{symmetries}:\textbf{A} we took $\Theta=0$ and varied $\theta$. For Figure \ref{symmetries}:\textbf{B} we fixed $\theta$ and varied $\Theta$. The plane of reflection corresponds to $\Theta = 0$. See~\cite[Fig. 5]{GhoshDuitsArxiv} for an intuitive explanation of the relation of $\Theta$ with respect to the momentum variables.
\section{Numerical Analysis of Problem $\Pcurve$ \label{ch:num}}

\subsection{Numerical Computations of the Jacobian of Exponential Map}\label{ch:Jac}
In this section we provide a numerical investigation into the absence of conjugate points on sub-Riemannian geodesics associated to the problem $\Pcurve$. Recall that a conjugate point is a critical value of the exponential map (cf.~Definition~\ref{def:exp}), i.e. at such a point one has $\det\left(\frac{\partial(Exp(\h(0),L))}{\partial(\h(0),L)}\right) = 0$.

Denote by $J$ the Jacobian of the exponential map, i.e.
$$J = \det\left(\frac{\partial(Exp(\h_1(0),\h_2(0),\h_4(0),\h_5(0),L))}{\partial(\h_1(0),\h_2(0),\h_4(0),\h_5(0),L)} \right).$$
To compute the Jacobian numerically, we approximate the partial derivatives by finite differences:
{\small
\begin{eqnarray*}
&\displaystyle \frac{\partial(Exp(\h_1(0),\h_2(0),\ldots,L))}{\partial(\h_1(0))} \approx \frac{Exp(\h_1(0)+ \Delta,\h_2(0),\ldots ,L)-Exp(\h_1(0)- \Delta,\h_2(0),\ldots,L)}{2 \Delta },\\
&\displaystyle \cdot \cdot \cdot\\
&\displaystyle \frac{\partial(Exp(\h_1(0),\h_2(0),\ldots,L))}{\partial L} \approx \frac{Exp(\h_1(0),\h_2(0),\ldots ,L+ \Delta)-Exp(\h_1(0),\h_2(0),\ldots,L- \Delta)}{2 \Delta }.
\end{eqnarray*}
}
We verified that the Jacobian is always positive for a million random points within the domain $\mathcal{D}_0$ of the exponential map of $\Pcurve$, recall (\ref{domain}). The points were determined as follows. The first integrals $\gc^2 = \sum_{i=1}^{3}\h_i^2(0)$ and $\sum_{i=3}^{5} \h_i^2(0) = 1$ allow us to introduce coordinates
\begin{equation*}
\begin{cases}
\h_4(0) = \sin\phi_2 \cos\phi_1,\\
\h_5(0) = \sin\phi_2 \sin\phi_1,\\
\h_3(0) = \cos\phi_2,
\end{cases} \qquad
\begin{cases}
\h_1(0) = r \cos\phi_3,\\
\h_2(0) = r \sin \phi_3,
\end{cases} \quad
\textrm{ where } r = \sqrt{\gc^2 -  \cos^2\phi_2} = \sqrt{\gc^2-\h_3^2(0)}.
\end{equation*}
By the rotational symmetry presented in Corollary~\ref{AllSym}, we can reduce one parameter by setting $\phi_1 = 0$. Furthermore, we recall that $\h_3(0) \geq 0$, which implies $-\frac{\pi}{2}\leq \phi_2 \leq \frac{\pi}{2}$. Thus, we can parameterize the domain of exponential map by
$\phi_2 \in [-\frac{\pi}{2}, \frac{\pi}{2}], 
\phi_3 \in [0, 2 \pi], 
\gc \in [\cos\phi_2, +\infty), 
L \in (0, \smax].$
We consider $\gc \in [\cos\phi_2, 10]$, and compute the Jacobian in both a random and a uniform grid on $(\phi_2, \phi_3, \gc, L)$. Here the restriction of $\gc$ from above is not crucial, as $\smax \to 0$ when $\gc \to \infty$. Furthermore, the absence of conjugate points for short arcs of geodesics follows from general theory (see~\!\cite{notes}). Finally, by Corollary~\ref{cor:maxsmax} we have $\smax \leq \frac12 \log\frac{(1+\gc)^2}{1 + \gc^2}$, that implies $\smax < 0.1$ for $\gc > 10$.

In Figure~\ref{fig:jacobian} we show several trajectories of different types (U-shaped curves for $\gc <1$ and S-shaped curves for $\gc>1$) and corresponding plots of the Jacobian for $s \in [0,\smax]$. Remarkably, the Jacobian is not just positive, it is even a monotonically increasing function of $s$ for the range of the plot. A similar behavior for the Jacobian can be seen on the closely related problem $\Pcurve$ on $\R^2$ (see.~\cite{DuitsSE2}), where the absence of conjugate points was proved.

\begin{figure}[ht]
\centering
\hspace{0.05\linewidth}
\begin{minipage}[c]{0.2\linewidth}
\centering
\includegraphics[width=\textwidth]{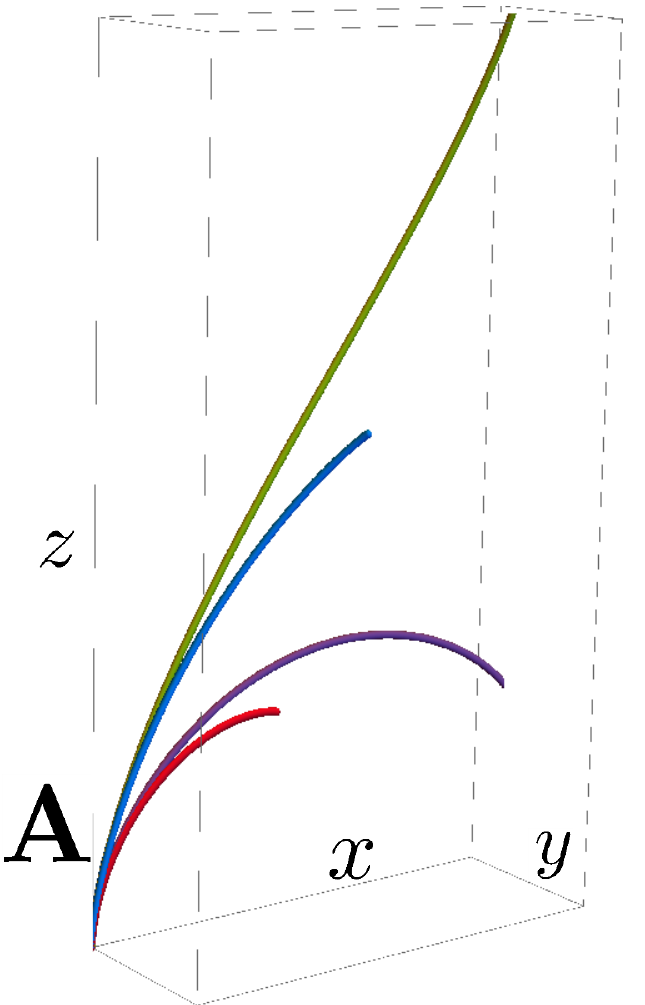}
\end{minipage}
\hspace{0.05\linewidth}
\begin{minipage}[c]{0.5\linewidth}
\centering
\includegraphics[width=\textwidth]{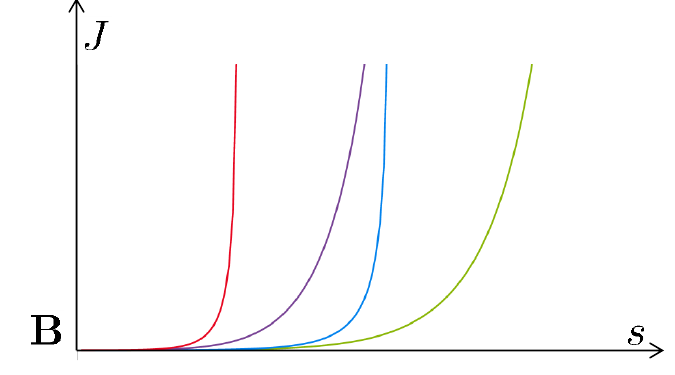}
\end{minipage}
\centering
\caption{\textbf{A:} Cuspless projections of sub-Riemannian geodesics in problem $\Pmec$ of different types. U-curves are depicted in green and blue colors, and S-curves are depicted in red and purple.
\textbf{B:} Plot of the Jacobian of exponential map, corresponding to these geodesics. We see that the Jacobian is positive (even increasing) for all $s \in (0,\smax)$. This supports our conjecture that conjugate points are absent before the first cusp point. }
\label{fig:jacobian}
\end{figure}
\subsection{The Range of the Exponential Map $\Pcurve$}\label{ch:conj}

There are a number of restrictions on the possible terminal points reachable by sub-Riemannian geodesics of problem $\Pmec$ with cuspless spatial projection. Together such points form the range $\mathcal{R}$ of the exponential map of $\Pcurve$, recall (\ref{eq:rangePcurveSE3}). We present some special cases which help us to get an idea about the range of the exponential map of $\Pcurve$. Recall that Corollary~\ref{cor:10} gave us the possible terminal positions (at $z=0$) when the final direction is opposite to the initial direction.

Based on our numerical experiments, we pose the following conjecture which is analogous to a result in the 2D case of finding cuspless sub-Riemannian geodesics in $(\SEtwo,\tilde{\Delta},\tilde{G}_{1})$~\cite{DuitsSE2,Boscain2013a}.
\begin{con}\label{Conjecture}
Let the range of the exponential map defined in Definition \ref{def:exp} be denoted by $\mathcal{R}$ and let the domain $\mathcal{D}_0$ of the exponential map be given by~(\ref{domain}).

\begin{itemize}[leftmargin=25pt]
 \item $Exp : \mathcal{D}_0\rightarrow \mathcal{R}$ is a homeomorphism when $\mathcal{D}_0$ and $\mathcal{R}$ are equipped with the subspace topology.
 \item $Exp : int(\mathcal{D}_0)\rightarrow int(\mathcal{R})$ is a diffeomorphism. Here $int(S)$ denotes the interior of the set $S$.
\end{itemize}

The boundary of the range is given by $\partial\mathcal{R} = S_B \cup S_R \cup S_L$ with
\begin{align}\label{all}
                S_B &= \{Exp(\blambda(0),\smax(\blambda(0))) \, | \, \blambda(0)\in \mathcal{D}_0\} \mbox{ and }\\ \nonumber
                S_R &= \{Exp(\blambda(0),s) \, | \, \blambda(0)\in \mathcal{D}_0 \mbox{ and } \lambda_4(0)^2 + \lambda_5(0)^2 =1 \mbox{ and } s>0\}\\ \nonumber
                S_L &= \{(\mathbf{0},R)\in \SE \, | \, R\be_z\cdot\be_z\geq 0\}.
\end{align}
\end{con}
Note that $S_B \in \mathcal{R}$ and $S_R \in \mathcal{R}$ but $S_L \not\in \mathcal{R}$.
\begin{figure}\centering
  \includegraphics[width=\hsize]{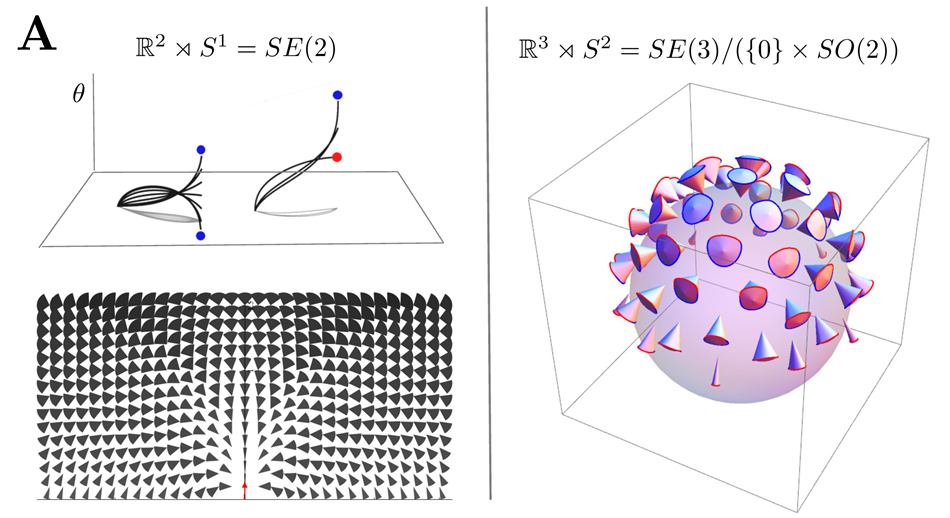}
\begin{subfigure}{.36\textwidth}
\centering
  \includegraphics[width=\textwidth]{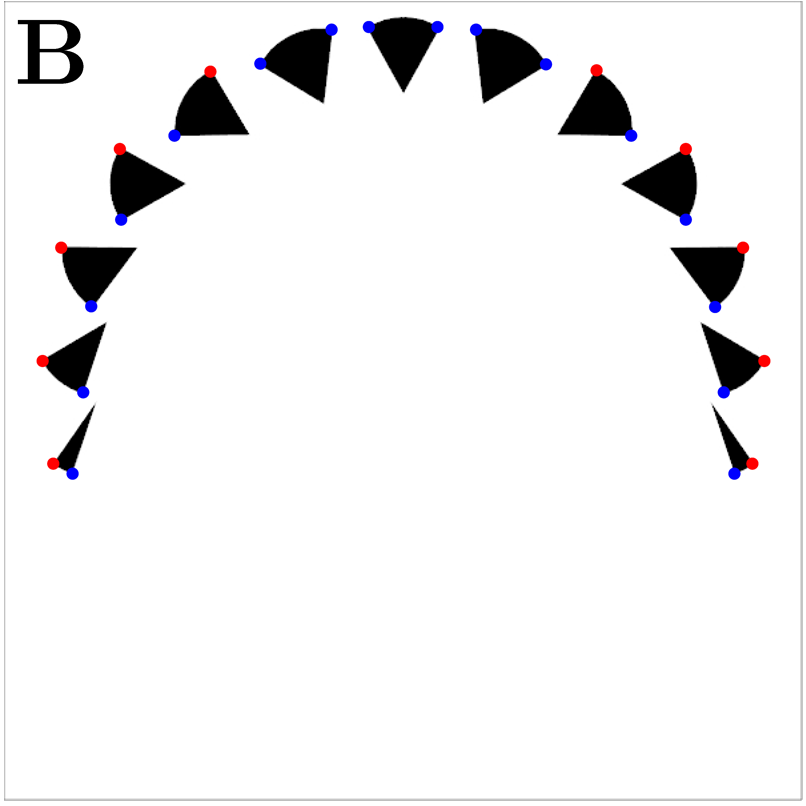}
\end{subfigure}%
\hspace{0.02\textwidth}
\begin{subfigure}{.28\textwidth}
\centering
  \includegraphics[width=\textwidth]{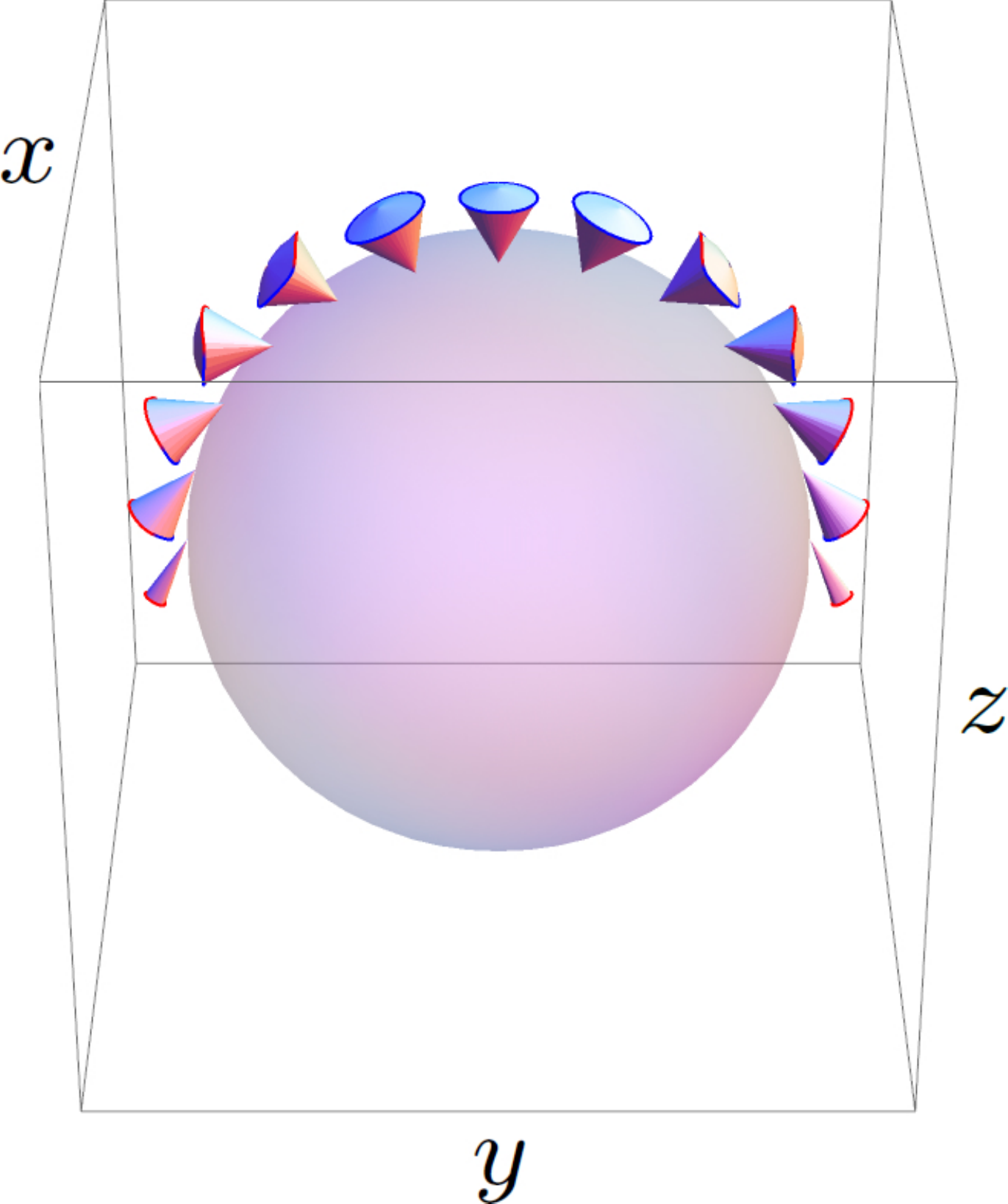}
\end{subfigure}
\caption{\textbf{A:} Comparison of the possible end conditions of $\mathbf{P_{curve}}$ for the 2D and 3D case. On the right, possible tangent directions of cuspless sub-Riemannian geodesics with unit length departing from the origin in the direction of $\be_z$ are depicted. In the $\SEtwo$ case (left) within $(\SEtwo,\tilde{\Delta},\tilde{G}_{1})$ studied in \cite{DuitsSE2}, this set of possible directions at each point is a connected cone \cite[Thm.6\&9]{DuitsSE2}. The boundary is obtained by end conditions of geodesics that either begin with a cusp point (shown in red) or end at a cusp point (shown in blue). \textbf{B:} Comparison in the special case when we set the end conditions on a half unit circle.}
\label{coneslice}
\end{figure}
\begin{figure}\centering
  \includegraphics[width=0.5\textwidth]{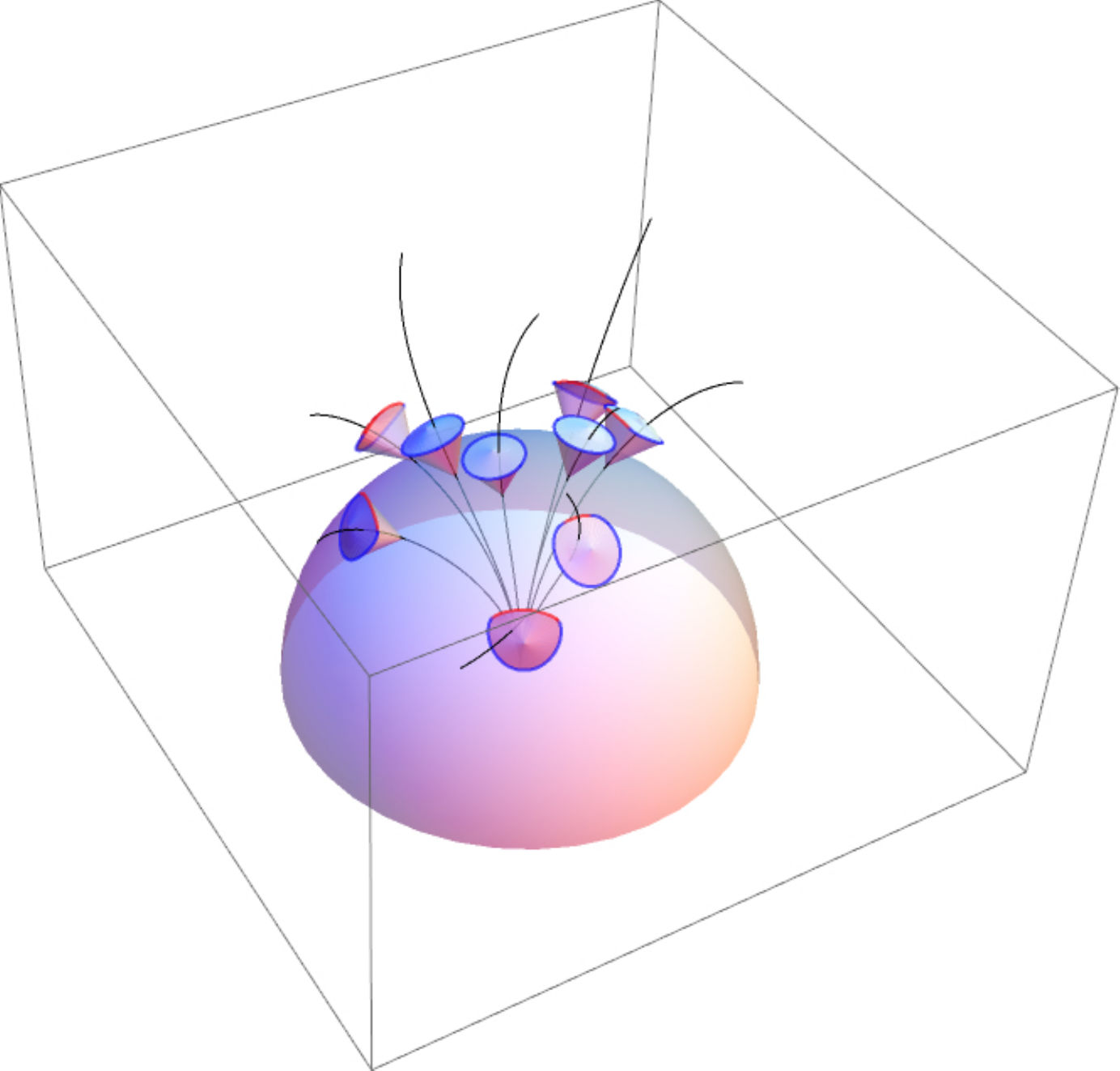}
\caption{The spatial part of arbitrary cuspless geodesics in $(\SE,\Delta,\mathcal{G})$ and the cones of reachable angles as depicted in Figure \ref{coneslice}. Note that the cuspless geodesics are always contained within the cones. We checked this for many more cases, which supports our Conjecture \ref{Conjecture}.}
\label{coneplot_g}       
\end{figure}
This conjecture would imply that no conjugate points arise within $\mathcal{R}$ and the problem $\mathbf{P_{curve}}$ \eqref{goal2} is well posed for \emph{all} end conditions in $\mathcal{R}$.
The proof of this conjecture would be on similar lines as in Appendix F of \cite{DuitsSE2}. If the conjecture is true, we have a reasonably limited set of possible directions per given end positions for which a cuspless sub-Riemannian geodesic of problem $\Pcurve$ exists. Then the cones of admissible end conditions for $\Pcurve$ (recall Definition~\ref{DEF:AM}) in Figure~\ref{coneslice}(A) form the image of the boundary of the phase space of $\{\ulh{2}(0),\ulh{1}(0)\}$ under the exponential map. Recall Figure~\ref{phaseplot}. These cones represent the boundary of the possible reachable angles by stationary curves of problem $\Pcurve$. Figure~\ref{coneslice}(B) shows the special case of the end conditions being on a unit circle containing the $z$-axis. The final tangents are always contained within the cones at each position. Numerical computations indeed seem to confirm that this is the case (see Figure \ref{coneplot_g}). The blue points on the boundary of the cones correspond to $S_B$ while the red points correspond to $S_R$ given in Eq.~\!(\ref{all}).
\subsection{Solving the Boundary Value Problem Associated to $\Pcurve$}  \label{ch:BVP}
Using explicit formulas for sub-Riemannian geodesics obtained in Theorem~\ref{statcurv} in Section~\ref{sec:ExplictFormulas} we developed a \emph{Wolfram Mathematica} package, that numerically solves the boundary value problem (BVP) associated to $\Pcurve$. The package is available by link \url{http://bmia.bmt.tue.nl/people/RDuits/final.rar}.
Note, that the BVP can be also tackled by a software \texttt{Hampath}~\cite{hampath} developed to solve optimal control problems. \texttt{Hampath} is based on numerical integration of a Hamiltonian system of PMP and second order optimality conditions. In contrast to \texttt{Hampath} our program does not involve numerical integration, and is based on numerical solution of a system of algebraic equations in Theorem~\ref{statcurv}, and relies on shooting based on the exact formulas, for details see~\cite{GhoshDuitsArxiv}.
\section{Conclusion}
In this article, we have derived explicit exact formulas of the geodesics of problem $\Pcurve$ in Theorem~\ref{statcurv} and Corollary~\ref{cor:aftermain}, where because of a scaling homothety we can set $\ksi=1$. We have shown in Theorem~\ref{lemma:One} that they are spatial projections of special cases of sub-Riemannian geodesics within $(\SE,\Delta, \mathcal{G}_{\ksi=1})$ whose spatial projections are cuspless. We have characterized the set $\mathcal{R}$ of admissible end conditions for problem $\Pcurve$ in Corollary~\ref{cor:main}. In Theorem \ref{th:integrability}, we have proved Liouville integrability for the corresponding sub-Riemannian problem.
In Theorem~\ref{sMax}, we have computed the first cusp time $t_{cusp}^{1}=t(\smax)$ of the sub-Riemannian geodesics explicitly.
We have shown the following geometric properties of geodesics of $\Pcurve$:
\begin{itemize}
\item global bounds on torsion in Theorem~\ref{ThVM},
\item they are planar if and only if the boundary conditions are coplanar, cf.~\!Theorem~\ref{coplan},
\item planar geodesics are globally optimal, cf.~\!Corollary~\ref{cor:embedSe2},
\item cuspless geodesics do not self intersect or roll up, cf.~\!Corollary~\ref{cor:mon},
\item they stay in the half space prescribed by the orientation of the initial condition (i.e.~$z\geq 0$ if $\ul{n}_{0}=\ul{e}_{z}$),
    which was formally shown for most cases (cf.~\!Corollaries~\ref{th:A},\ref{cor:zpw0} and \ref{th:B}).
    Also we analyzed cases where the plane $z=0$ is reached,
    cf.~Corollary~\ref{th:B},
\item their rotational and reflectional symmetries in Theorem~\ref{AllSym}.
\end{itemize}
Finally, we provided a numerical \emph{Mathematica} package to solve the boundary value problem via a shooting algorithm. We included numerical support for the expected absence of conjugate points on the sub-Riemannian geodesics (with $\h_{6}(0)=0$) with cuspless spatial projections, i.e. $t_{conj}\geq t_{cusp}^{1}$. Also, numerical computations on the cones of reachable angles (and their boundaries, cf.~Figure~\ref{coneslice}) seem to reveal the same homeo/diffeomorphic properties of the exponential map integrating the canonical ODE's in Pontryagin maximum principle, that were formally shown on the $\SEtwo$ case \cite[Thm.6\&9]{DuitsSE2}. In future work, we plan to analyze the sub-Riemannian spheres via viscosity solutions of sub-Riemannian HJB-equations, i.e. extend the work~\cite{bekkers} to the $\SE$ case. This will yield a numerical computation of the sub-Riemannian spheres and the 1st Maxwell-set.
\section*{Acknowledgments}
The research leading to these results has received funding from the European
Research Council under the European Community's Seventh Framework Programme
(FP7/2007-2013) / ERC grant \emph{Lie Analysis}, agr.~nr.~335555.
The authors gratefully acknowledge EU-Marie Curie project \emph{MANET}
ag.~no.607643 and the European Commission ITN-FIRST, agreement No. PITN-GA-2009-238702 for financial support. We gratefully acknowledge Prof. Yu.L. Sachkov for suggestions on the Liouville integrability and the co-adjoint orbit structure, and for many other suggestions on the article
and its exposition. We thank Dr.ir. A.J.E.M. (Guido) Janssen for fruitful discussions on  elliptic integral representation in Theorem~\ref{statcurv}, and we thank the anonymous reviewers for many valuable suggestions.
\appendix
\section{Proof of Theorem~\ref{lemma:One}}\label{app:A}
Here we rely on the formulation of problem $\PMEC$ using the control variables as given by~(\ref{eq:PMECStatControl}) in Section~\ref{sec:PMEC}. To this end we note that $\langle \omega^i|_{\gamma}, \dot{\gamma} \rangle= u^i$ for $i \in \{3,4,5\}$.
So for $i=3$ we have $\langle \omega^3|_{\gamma},\dot{\gamma}\rangle >0 \desda u^3 >0$. Particularly, this holds for a smooth minimizer $\gamma=\gamma^*$ of problem $\PMEC$.

If the end-condition $g=g_1=(\ul{x}_1,R_{1})$ in (\ref{distance}) is chosen such that the optimal control $u_{3}(t)> 0$ for $t \in (0,T)$, then $\frac{ds}{dt}(t)>0$ and the minimizer is parameterizable
by spatial arclength $s$. Let $\gamma$ be a horizontal curve in $(\SE,\Delta, \mathcal{G}_{\xi})$.
We define $\overline{\gamma}(s):=\gamma(t(s))=(\ul{x}(s),\ul{n}(s))$, and $\overline{u}^{k}(s)=u^{k}(t(s))$ and
let us recall $\gamma$ is horizontal, i.e.
\[
\begin{array}{l}
\dot{\gamma}(t)= \sum \limits_{i=3}^{5} u^{i}(t) \left.\mathcal{A}_{i}\right|_{\gamma(t)}, \
\overline{\gamma}'(s)= 1 \left. \mathcal{A}_{3} \right|_{\overline{\gamma}(s)} + \overline{u}^{4}(s)
\left. \mathcal{A}_{4} \right|_{\overline{\gamma}(s)} + \overline{u}^{5}(s) \left. \mathcal{A}_{5} \right|_{\overline{\gamma}(s)}.
\end{array}
\]
Lifting of a curve $\overline{\ul{x}}(\cdot)$ to a curve $(\ul{x}(\cdot),\ul{n}(\cdot))$ into $\R^{3}\rtimes S^{2}$ is done
by setting $\ul{x}'(s)=\ul{n}(s)$.
Let $c_{i,j}^{k}$ denote the usual structure constants of the Lie algebra of $\SE$ (see Table~\ref{tabcomA}), then
{\small
\[
\ul{x}''(s)=\ul{n}'(s)= \frac{d}{ds} \ul{x}'(s)= \frac{d}{ds} \left.\mathcal{A}_{3}\right|_{\overline{\gamma}(s)}=
\sum \limits_{j,k=1}^{3} c_{j,3}^{k}\, \langle
\left.\omega^{j}\right|_{\overline{\gamma}(s)} , \overline{\gamma}'(s) \rangle  \left.\mathcal{A}_{k}\right|_{\overline{\gamma}(s)}=
-\overline{u}^{4}(s)\, \left.\mathcal{A}_{2}\right|_{\overline{\gamma}(s)} + \overline{u}^{5}(s)  \left.\mathcal{A}_{1}\right|_{\overline{\gamma}(s)}.
\]
}
Direct computation of the Frenet-Seret ODE along horizontal curves (cf.~\!\cite[ch:2,Thm3.16]{GhoshDuitsArxiv}) yields
the following
expressions for curvature magnitude, and torsion magnitude:
\[
\begin{array}{l}
|\kappa(s)|^2=|\overline{u}^{4}(s)|^2+ |\overline{u}^{5}(s)|^2, \
\tau(s)= \frac{\overline{u}^{4}(s) (\overline{u}^{5})'(s) - \overline{u}^{5}(s) (\overline{u}^{4})'(s)}{|\overline{u}^{4}(s)|^2+ |\overline{u}^{5}(s)|^2} .
\end{array}
\]
Furthermore, we have $\|(\ul{x}'(s)\|=\overline{u}^{3}(s)=1$, and we see that the energy functionals of $\mathbf{P_{curve}}$ and $\mathbf{P_{mec}}$ coincide, as we have
\begin{equation}\label{agree}
\int \limits_{0}^{T} \sqrt{\left.\mathcal{G}_{\xi}\right|_{\gamma(t)}(\dot{\gamma}(t),\dot{\gamma}(t)) } \, {\rm d}t = \int \limits_{0}^{L} \sqrt{\left.\mathcal{G}_{\xi}\right|_{\overline{\gamma}(s)}(\overline{\gamma}'(s),\overline{\gamma}'(s)) } \, {\rm d}s = \int \limits_{0}^{L} \sqrt{\kappa^{2}(s)+\xi^2} \, {\rm d}s.
\end{equation}
Application of PMP (scf.~\!Section~\ref{ch:PMP}) to $\PMEC$ 
yields the following ODE for the horizontal part
$$\dgamma = \h^3\cA_3|_{\gamma} + \h^4\cA_4|_{\gamma} + \h^5\cA_5|_{\gamma},$$
and for the vertical part, we obtain the ODE
\begin{equation*}\label{PreL}
\begin{array}{l}
\frac{\dif}{\dif t}\h_{i}= -\sum \limits_{b=3}^{5}\sum \limits_{l=1}^{6} c^{l}_{i,b} \h^{b} \h_{l}  \
\desda \\
\frac{\dif}{\dif t}(\h_1,\h_2,\h_3,\h_4,\h_5,\h_6) =
(\h_3\h_5, \h_3\h_4, \h_1\h_5 - \h_2\h_4,
\ksi^{-2}\h_3\h_2 -\h_5 \h_6, -\ksi^{-2}\h_3\h_1+\h_4 \h_6,0).
\end{array}
\end{equation*}
Note that \emph{reciprocal} momentum components are related by the inverse metric tensor and thereby given by $\h^{3}=\xi^{-2}\h_3$, $\h^{4}=\h_4$, $\h^{5}=\h_5$.
PMP gives us that the stationary curves obtained via these ODE's are short time local minimizers.
It also provides us the Hamiltonian $H(\h)= \frac{1}{2}\left(\xi^{-2}\h_{3}^{2} + \h_{4}^{2} +\h_{5}^{2}\right)$
and the Exponential map.

Now we must choose $\gamma(L)=g_{1} \in \SE$ from the equivalence class $[g_1]=\{g \in \SE \;|\; g \sim g_{1}\}$ (i.e. left-coset recall (\ref{leftcoset}))
such that the minimum in (\ref{onthequotient}) is attained.  This does not hold for all elements in $\SE$.
In fact it only holds for those end conditions that can be reached with geodesics having $\h_{6}(0)=0$. This follows
from the fact that along all sub-Riemannian geodesics one has $\dot{\h}_6=0$ and the fact that the sub-Riemannian minimizers
with $\h_{6}=0$ are precisely the ones where the constraint $\omega^{6}=0$ is redundant and the result follows.
See also Figure~\ref{fig:lemma}
$\hfill \Box$.
\section{Cartan connection $\overline{\nabla}$ on sub-Riemannian manifold $(\SE, \Delta, \mathcal{G}_{\xi})$}\label{app:Cartan}
In this appendix we provides background/embedding of Definition~\ref{def:CC} and Theorem~\ref{th:2}, and in particular Eq.~\!(\ref{CC}) and Eq.~\!(\ref{imp1}), into theory of Cartan connections.

The $-$ Cartan connection \cite{CS} is induced by the Maurer-Cartan form $(L_{g^{-1}})_*$ which induces a Cartan-Ehresmann connection on the principal G-bundle $P = (\SE, \{e\}, \pi, R)$, with total space $\SE$, base space $\{e\}= \SE/\SE$, projection $\pi(g) = e$ and the right action $R_{g_1}g_2 = g_2 g_1$.

The construction is as follows. The Maurer-Cartan form induces a connection $\tilde{\omega}$ on the associated vector bundle $\SE \times_{Ad} \mathcal{L}(\SE)$, where $\mathcal{L}(\SE)$ denotes the Lie algebra of left-invariant vector fields,  given by
$ \tilde{\omega} = \sum_{j=1}^{6} (Ad)_{\ast} \cA_j \otimes \omega^{j} = \sum_{i,j,k=1}^{6} c_{i,j}^{k} \cA_k \otimes \omega^i \otimes \omega^j$.
The form $\tilde{\omega}$ induces a matrix-valued 1 form $-\tilde{\omega}(\omega^k, \cdot, \cA_j)$ on the frame bundle, and moreover it induces a connection $\nabla$ on tangent bundle $T(\SE)$, where $\nabla_{(\sum_{i=1}^{6} \dot{\gamma}^i \cA_i)} (\sum_{k=1}^{6}a^k \cA_k) = \sum_{k=1}^{6}\left( \dot{a}^k - \sum_{i,j=1}^{6} c_{i,j}^k \dot{\gamma}^i a^j\right) \cA_k.$
This is all still in the Riemannian setting.

In the sub-Riemannian setting of $(\SE,\Delta,\mathcal{G}_{\xi})$, one relies on a different structure subgroup $\widetilde{\SEtwo}$ (consisting of translations and rotations in the $xy$-plane only) isomorphic to $\SEtwo$, rather than structure group $\SE$ in the Riemannian setting.
This boils down to constraining
some of the summation indices and therefore we use $\overline{\nabla}$ given by~(\ref{CC}) instead of $\nabla$. 
Next, we explain how partial connection $\overline{\nabla}$ appears in Cartan geometry.

In the sub-Riemannian manifold $(\SE, \Delta, \mathcal{G}_{\xi})$, with $\Delta={\rm ker}\{\omega^1\}\bigcap {\rm ker}\{\omega^2\}\bigcap {\rm ker}\{\omega^6\}$,
the directions $\cA_1$, $\cA_2$ and $\cA_6$ are prohibited. To get a better grasp on what this means on the manifold level, we consider principal fibre bundles.
To this end, we consider the subgroup isomorphic to $\SEtwo$ given by $\widetilde{\SEtwo}=\{\exp{(c^1A_1 + c^2A_2 + c^6A_6)}|c^1,c^2, c^6\in \R\}$ with $A_k=\cA_k|_e$.

Now we consider the principal fibre bundle $\overline{P}=(\SE,\SE/\widetilde{\SEtwo},\pi,R)$ with $R_hg=gh$, $h\in \widetilde{\SEtwo}$, $\pi (g)=[g]= g \, \widetilde{\SEtwo} \in \SE/\widetilde{\SEtwo}$. On $\overline{P}$, we consider the Maurer-Cartan form $\bar{w}=(L_{h^{-1}})_*$, more precisely
$
\bar{w}(\cA_g) = \sum_{i=3}^{5}\langle\omega^i|_g,X_g\rangle\cA_i
$.

Via the group representation $\widetilde{\SEtwo}\ni h\mapsto Ad(h):= (L_{h^{-1}}R_h)_*$, we obtain the associated vector bundle (def. 3.7 in \cite{DuitsAMS2}) $(\SE\times_{Ad|_{\widetilde{\SEtwo}}}\cL(\SE))$ with corresponding connection form
\[
\bar{w}=\sum_{j=3}^{5}\left(Ad|_{\widetilde{SE(2)}}\right)_*(\cA_j)\otimes\omega^j = \sum_{j=3}^{5}ad(\cA_j)\otimes\omega^j = \sum_{i,j,k=3}^{5}c^k_{i,j}\cA_k\otimes\omega^i\otimes\omega^j,
\]
where $ad(X)=[\cdot,X]$.
This yields a $3\times 3$ matrix valued one form on the frame bundle
$
\bar{w}^k_j = -\bar{w}(\omega^k,\cdot ,\cA_j)
$
which yields a partial connection on the horizontal part $\Delta$ of $T(SE(3))$:
\begin{align}\label{lazy}
\overline{\nabla}_X\cA &= \sum_{k=3}^{5}\left(\dot{a}^k\cA_k + \sum_{j=3}^{5}a^j \bar{w}^k_j(X)\cA_k\right)
= \sum_{k=3}^{5}\left(\dot{a}^k\cA_k + \sum_{i,j=3}^{5}a^j \dot{\gamma}^i c^k_{j,i}\cA_k\right),
\end{align}
with $X=\sum_{i=3}^{5}\dot{\gamma}^i\cA_i$, $\cA=\sum_{k=3}^{5}a^k\cA_k$ and $\bar{w}^k_j(\cA_i) = -\bar{w}(\omega^k,\cA_i,\cA_j) = -c^k_{i,j}$ where the Christoffels are equal to minus the structure constants of the Lie algebra and where $\dot{a}^k = \sum_{i =3}^{5}\dot{\gamma}^i(\cA_i|_{\gamma}a^k)$.

Finally, Eq.~\!(\ref{lazy}) is equivalent to Eq.~\!(\ref{CC}), as $c^{k}_{i,j}=-c^{k}_{j,i}$. As shown in the proof of Theorem~\ref{th:2}
partial connection $\overline{\nabla}$ on the tangent bundle induces a partial connection $\overline{\nabla}^{*}$ on the cotangent bundle (given by the left hand side of (\ref{imp1})).




%
%

\end{document}